\documentclass{amsart}
\usepackage{amsmath,amsthm,amssymb}

\begin{document}

\theoremstyle{plain}	
\newtheorem{thm}{Theorem}[section]
\newtheorem{lemma}[thm]{Lemma}
\newtheorem{proposition}[thm]{Proposition}
\newtheorem{corollary}[thm]{Corollary} 

\theoremstyle{definition}
\newtheorem{defn}{Definition}[section]

\theoremstyle{remark}
\newtheorem{example}{Example}[section]
\newtheorem{remark}{Remark}[section]
\newtheorem{conjecture}{Conjecture}[section]
\newtheorem{remarks}{Remarks}[section]
\newtheorem*{problem}{Problem}
\newtheorem*{notation}{Main notation}

\newcommand{\arctg}{\operatorname{arctg}}
\newcommand{\sech}{\operatorname{sech}}
\newcommand{\ve}{\varepsilon}
\newcommand{\Ltau}{L^2_{\tau}(\mathbb{R})}
\newcommand{\Lltau}{L^1_{\tau}(\mathbb{R})}
\newcommand{\lu}{\tilde{u}}
\newcommand{\lv}{\tilde{v}}
\newcommand{\dt}{\lVert t \rVert}
\newcommand{\dtau}{\lVert \tau \rVert}
\newcommand{\supp}{\operatorname{supp}}

\numberwithin{equation}{section}

\title[Ergodic attractors of scalar semilinear PDE's]{Ergodic attractors and almost-everywhere asymptotics of scalar semilinear parabolic differential equations}

\author{Sini\v{s}a Slijep\v{c}evi\'{c}, Zagreb}

\date{November 17th, 2017}

\begin{abstract}We consider dynamics of scalar semilinear parabolic equations on bounded intervals with periodic boundary conditions, and on the entire real line, with a general nonlinearity $g(t,x,u,u_x)$ either not depending on $t$, or periodic in $t$. While the topological and geometric structure of their attractors has been investigated in depth, we focus here on ergodic-theoretical properties. The main result is that the union of supports of all the invariant measures projects one-to-one to $\mathbb{R}^2$. We rely on a novel application of the zero-number techniques with respect to evolution of measures on the phase space, and on properties of the flux of zeroes, and the dissipation of zeroes. As an example of an application, we prove uniqueness of an invariant measure for a large family of considered equations which conserve a certain quantity ("mass"), thus generalizing the results by Sinai for the scalar viscous Burgers equation.
\end{abstract}

\vspace{2ex}

 
\maketitle

%

\tableofcontents

\part{Preliminaries}

\section{Introduction}

We consider the following equation:
\begin{equation}
u_{t}=u_{xx}+g(t,x,u,u_{x}),  \label{r:main}
\end{equation}%
where $g$ satisfies the usual conditions guaranteeing local existence of solutions, given as (A1-3) below. In particular, we assume that $g$ is periodic in $t,x$, and consider solutions on the entire real line, without the assumption of decay to $0$ at infinity (the {\it extended, time-periodic case}). For the sake of completeness, we also cover less general cases of $g$ not depending on $t$ (the {\it autonomous case}), and for $x \in \mathbb{S}^1$ (the {\it bounded case}). A more precise setting is given below. For brevity, we frequently denote the considered cases with letters E or B (for extended vs. bounded), and P or A (time-periodic vs. autonomous).

We first briefly recall here results on geometric and topological dynamics of (\ref{r:main}). The asymptotics of (\ref{r:main}) on the bounded domain with separated boundary conditions has been studied in detail (see \cite{Joly10,Polacik:02} and references therein) and is reasonably well-understood. In particular, under assumptions (A1-3), for any global, uniformly bounded orbit, the $\omega$-limit set contains a single orbit (equilibria in the autonomous or a periodic orbit in the periodic case) (\cite{Polacik:02}, Theorem 4.2 and references therein). With periodic boundary conditions, i.e. in our setting in the B/A case and assuming (A1-3), Fiedler and Mallet-Paret \cite{Fiedler89} have shown that the $\omega$-limit set of any global, bounded solution projects to a plane, and then has the structure in accordance to the Poincar\'{e}-Bendixson theorem. That means that it consists of a single periodic orbit, or of equilibria and connecting (homoclinic and heteroclinic) orbits. Tere\v{s}\v{c}\'{a}k \cite{Terescak:94} has shown that in the B/P case, assuming (A1-3), the $\omega$-limit set of any global, bounded orbit also projects injectively and continuously into $\mathbb{R}^2$. The structure of the $\omega$-limit set can then be much more complex, as shown by Fiedler and Sandstede \cite{Fiedler:92,Sandstede:91}.

The structure of the attractor of (\ref{r:main}) on the bounded domain with separated or periodic boundary conditions in the autonomous case is as follows: the attractor is then generically Morse-Smale, and can in many cases be classified by the graph structure of the equilibria and their connections (\cite{Fiedler:12,Fiedler:14,Joly10,Polacik:02} and references therein). Similar questions in the B/P case, and the extended case seem to be currently beyond reach. When assuming decay to $0$ at infinity, the dynamics in some cases (for example for $g$ not depending on $x,u_x$ \cite{Feireisl:00}) is similar to the dynamics on the bounded domain with separated boundary conditions, i.e. uniformly bounded orbits then converge to a single periodic solution. If there is no decay to zero at infinity, the attractor seems to be typically infinite dimensional (assuming sufficiently weak topology so that uniformly locally bounded orbits are relatively compact, see Section \ref{s:prelim}), and the asymptotics can be very complex even in the "extended gradient case" (see \cite{Polacik:15} and references therein, also Subsection \ref{ss:AllenCahn}).

While the ergodic theory of PDE's has received much less attention than the topological and geometrical perspective, it is a current area of research which is both physically and mathematically relevant to the dynamics of (\ref{r:main}).

Relevant and related recent ergodic-theoretical results include for example an extension of the notion of SRB measures to PDE's by Blumenthal and Young \cite{Blumenthal:15}, and results on almost-everywhere global existence of solutions with respect to a 'natural' measure e.g. by Nahmod, Pavlovi\'{c} and Staffilani for the Navier-Stokes equation \cite{Nahmod:12}, building on an approach of Bourgain \cite{Bourgain:96}. Specifically with regards to the equation (\ref{r:main}), Gallay and the author have shown that if the equation has in addition a formally gradient structure, then the invariant measures are supported on the set of equilibria \cite{Gallay:01, Slijepcevic:99, Slijepcevic00}. Zelik \cite{Zelik:03} has in the same case deduced that the topological entropy is thus 0.

In this paper we show that contrary to potentially very complex topological and geometric structure of the attractor, from the point of view of ergodic theory the dynamics of (\ref{r:main}) is in all the considered cases relatively simple. Specifically, we investigate the structure of the set of invariant (Borel probability) measures of (\ref{r:main}) on the phase space. In particular, we analyse the union of supports of all the invariant measures, a set which is a subset of the global attractor which we propose to call {\it ergodic attractor}. In all the considered cases, we show that the ergodic attractor projects one-to-one to $\mathbb{R}^2$ (subject to a technical restriction of finite average density of zeroes in the extended case, which we believe to be generically true and likely redundant), and that in many cases it is one-dimensional.

The dynamical relevance (and physical interpretation) of this is as follows: in the bounded case, the ergodic attractor contains all {\it $\omega$-limit sets on average} of all relatively compact orbits (Subsection \ref{ss:ergbounded}). The $\omega$-limit set on average has been proposed in the context of partial differential equations in \cite{Gallay:01}, and contains accumulation points of a relatively compact orbit for non-zero density of times. We argue that physically only these orbits are "observable" (Lemma \ref{l:observability}), thus the description of the ergodic attractor reasonably completely describes "observable" dynamics. In particular, the ergodic attractor contains any "chaos" if present \cite{Slijepcevic13a}. In the extended case, the ergodic attractor consists of "space-time observable" orbits (Subsection \ref{ss:ergextended}); contains the space-time chaos as constructed in \cite{Mielke09,Turaev10} if present \cite{Slijepcevic13a}; and frequently describes asymptotics of $\mu$-a.e. $u$ with respect to any Borel probability measure on the phase space invariant with respect to the spatial shift (see results for Burgers like equation below; also Subsections \ref{ss:extended} and \ref{ss:further}).

An example of an application of our results is a generalization of the results by Sinai \cite{Sinai:91} for the viscous, periodically forced Burgers equation:
\begin{equation}
u_t =u_{xx} - u\: u_x + \hat{g}(x,t), \label{r:burgers}
\end{equation}
where $\hat{g}$ is sufficiently smooth, periodic in $x$ and $t$, and such that for all $ t \in \mathbb{R}$, $\int_0^1 \hat{g}(x,t)dx = 0$. Sinai showed the following (extended to quasi-periodic forcing in \cite{Sinai:98}, higher dimensions on bounded domain and stochastic forcing in \cite{Sinai:96}, and to inviscid limit on bounded domain and stochastic forcing in \cite{E:00}): 

\begin{itemize}

\item[(i)] Firstly, it was established that there is a unique solution of (\ref{r:burgers}) periodic in $x$ and $t$, denoted by $v^0(t)$, such that for any initial condition $u \in H^{2\alpha}(\mathbb{S}^1)$, $\int_0^1 u(x)dx=0$, we have that $\lim_{t \rightarrow \infty}|u(x,t)-v^0(x,t)|=0$ (a pointwise convergence) (a special case of \cite{Sinai:91}, Theorem 1). 

\item[(ii)] Secondly, such asymptotics is shown to hold also on the extended domain {\it for a.e. initial condition with respect to some probability measure} on the phase space, as long as the probability measure satisfies certain conditions (see Section \ref{burgers} for details).

\item[(iii)] Thirdly, each probability measure from (ii) converges in weak$^*$ topology with respect to the induced semiflow on the space of measures to the Dirac measure concentrated on $v^0$.

\end{itemize}

The main technique in \cite{E:00, Sinai:91,Sinai:96,Sinai:98} is the Cole-Hopf transformation, and the integral representation of the transformed solutions. As already noted in \cite[p347]{Sinai:96}, the key property of (\ref{r:burgers}) is that $\int_0^1 u(x)dx$ is the invariant. We show here that such invariance (the condition (B3) below) in essence suffices to establish (i) and versions of (ii) and (iii). We assume in addition only certain weak dissipativity conditions (B1-2) ensuring global existence and boundedness of solutions. We do not use here the Cole-Hopf transformation. Instead, our main technique is an extension of the zero-number techniques to measures (see the next subsection). 

Finally, we argue that the techniques developed here also extend to the equation (\ref{r:main}) with an additional random force term such as for example considered in \cite{E:00,Sinai:98}, and also to discrete-space continuous-time, or discrete-space discrete-time 1d monotone systems without and with random force, as further discussed in Section \ref{s:open}. In particular, we hope that the main technique of the paper: the zero-function as a Lyapunov function with respect to evolution of measures induced by the dynamical system, can be useful in characterizing uniqueness of invariant measures, thus questions related to existence of physical and SRB measures in the deterministic case, and phase transitions in the random case of these models.

\section{Setting and statements of results}

\subsection{Setting and assumptions}
\label{ss:general} 
We first specify the function spaces on which we consider (\ref{r:main}). In the bounded case, we consider $\mathcal{X}^{\alpha}:=H^{2\alpha}(\mathbb{S}^1)$, where $\mathcal{X}:=L^2(\mathbb{S}^1)$, and $3/4 < \alpha < 1$ is such that $\mathcal{X}^{\alpha}$ is continuously embedded in $C^1(\mathbb{S}^1)$. In the {\it extended} case, the domain is the entire $\mathbb{R}$ without assuming decay to zero at infinity. The phase space is then the fractional uniformly local space $\mathcal{X}^{\alpha}:=H^{2\alpha}_{\text{ul}}(\mathbb{R})$, where $\mathcal{X}^2_{\text{ul}}(\mathbb{R})$,
$\alpha$ is as above (see Appendix \ref{s:fractional} for key facts on uniformly local spaces), and then $H^{2\alpha}_{\text{ul}}(\mathbb{R})$ is continuously embedded in  $C^1(\mathbb{R})$. 
The bounded case may be considered as an invariant subset of the extended case, as $H^{2\alpha}(\mathbb{S}^1)$ embeds naturally in $H^{2\alpha}_{\text{ul}}(\mathbb{R})$ as the invariant set of spatially periodic solutions. We denote by $S : \mathcal{X}^{\alpha} \rightarrow \mathcal{X}^{\alpha}$ the spatial shift $Su(x)=u(x-1)$ (identity in the bounded case).

The standing assumptions on the nonlinearity $g:(t,x,u,\xi) \mapsto g(t,x,u,\xi)$ are as follows:
\begin{itemize}
	
	\item[(A1)] $g$ is continuous in all the variables.
	
	\item[(A2)] $g$ is locally H\"{o}lder continuous in $t$ and locally Lipschitz continuous in $(u,\xi)$.
	
	\item[(A3)] $g$ is $1$-periodic in $x$ and $t$.
	
\end{itemize}
It is well-known that (A1-3) suffice for local existence of solutions in bounded and extended case to hold (Section \ref{s:prelim}). In addition, in the first part of the paper, we also assume:
\begin{itemize}
	\item[(A4)] There exists a set $\mathcal{B}$, closed and bounded in $\mathcal{X}^{\alpha}$-norm, $S$-invariant in the extended case, such that if $u_0\in \mathcal{B}$, $t_0 \in \mathbb{R}$, and the solution of (\ref{r:main}), $u(t_0)=u_0$ exists on $(t_0,t_1)$, then for all $t \in (t_0,t_1)$ we have that $u(t)\in \mathcal{B}$. 
\end{itemize}

As recalled in Section \ref{s:prelim}, conditions (A1-4) suffice for global existence and uniqueness of solutions of (\ref{r:main}) to hold. Furthermore, in the autonomous case (\ref{r:main}) generates a continuous semiflow on $\mathcal{B}$ denoted by $T(t)$, $t \geq 0$. In the periodic case, the time-one map $T:\mathcal{B} \rightarrow \mathcal{B}$ is continuous. Because of (A3) we have that $S$ and $T(t)$, resp. $T$ commute; and $S$ is continuous.

\begin{remark}
Sufficient conditions in various contexts for (A4) to hold are given in \cite{Lunardi:95}, Section 7 (see also \cite{Polacik:02} and references therein). These results also apply in the extended case, in the view of the comments in the Appendix \ref{s:fractional}. 
\end{remark}

The notion of {\it invariance} throughout the paper will depend on the considered case: unless otherwise specified, an {\it invariant set} will be any set invariant with respect to all the actions in the Table \ref{t:table}:
\begin{equation}
\begin{array}{lcc}
\text{Actions:} & \text{Autonomous (A)} & \text{Time-periodic (P)} \\ 
\text{Bounded (B)} & T(t), t\geq 0; & T \\ 
\text{Extended (E)} & T(t), \: t\geq 0 ;\: S & T;S\text{.}%
\end{array}
\label{t:table}
\end{equation}%
We always consider $\omega$-limit sets with respect to the semiflow $T(t)$, $t\geq 0$ in the autonomous case, and for the sequence of maps $T^n$, $n \in \mathbb{N}$ in the time-periodic case.
In the extended case, we will equip $\mathcal{X}^{\alpha}$ with a coarser topology, to ensure that all the orbits bounded in $\mathcal{X}^{\alpha}$ are relatively compact, so that we can consider asymptotics and invariant measures (see Section \ref{s:prelim} and Appendix \ref{s:fractional} for the choice of topology and a discussion). We define an {\it invariant measure} to be a Borel probability measure on $\mathcal{B}$, invariant with respect to all the actions in Table \ref{t:table}.

\subsection{Statements of the results: ergodic Poincar\'{e}-Bendixson theorems}
Denote by $\mathcal{E}$ the ergodic attractor, i.e. the union of supports of all the invariant measures. As $\mathcal{E}$ depends on the choice of $\mathcal{B}$ in (A4), we may occasionally write $\mathcal{E}(\mathcal{B})$; the argument $\mathcal{B}$ will be omitted when the chosen $\mathcal{B}$ is clear from the context. The main result in the bounded case is that the set $\mathcal{E}$ is not too large, i.e. that it is at most two dimensional:

\begin{thm} \label{t:main1} {\bf Ergodic Poincar\'{e}-Bendixson Theorem.}
Assume (A1-4) holds in the bounded case. Then $\mathcal{E}$ projects continuously and one-to-one to $\mathbb{R}^2$, with the projection $\pi : \mathcal{E} \rightarrow \mathbb{R}^2$ given with 
\begin{equation}
\pi(u_0)=(u_0(0), (u_0)_x(0)). \label{d:pi}
\end{equation}
\end{thm}

In the B/A case, this already follows from Fiedler and Mallet-Paret Poincar\'{e}-Bendixson theorem \cite{Fiedler89} (see Subsection \ref{ss:BDC} for further comments). In the B/P case, it seems new, and is complementary to the results of Tere\v{s}\v{c}ak \cite{Terescak:94}.

To establish an analogous result in the extended case, we require a technical condition of non-degeneracy of $\mathcal{E}$, by which we mean that the average density of zeroes on $\mathcal{E}$ is bounded. It is rigorously given in Definition \ref{d:nondegenerate}; we note here that it suffices that for any two $u_0,v_0 \in \mathcal{E}$, 
\begin{equation}
\liminf_{n \rightarrow \infty} \frac{1}{2n}\sum_{k=-n}^{n-1} z(S^ku_0,S^kv_0)< \infty, \label{d:density}
\end{equation}
where $z(u_0,v_0)$ is the number of zeroes of $u_0(x)-v_0(x)$ for $x\in [0,1)$ (a precise definition is given by (\ref{d:ZFD}) and (\ref{d:zfd})).

\begin{thm} \label{t:main1b} {\bf Extended Ergodic Poincar\'{e}-Bendixson Theorem.}
	Assume (A1-4) holds in the extended case, and assume that $\mathcal{E}$ is non-degenerate. Then $\mathcal{E}$ projects continuously and one-to-one to $\mathbb{R}^2$, with the projection $\pi : \mathcal{E} \rightarrow \mathbb{R}^2$ given with (\ref{d:pi})
\end{thm}

\begin{remark} Non-degeneracy of $\mathcal{E}$ is expected to hold generically, and possibly always. This follows from the results of Angenent and Chen \cite{Angenent:88,Chen:98}: as $\mathcal{E}$ consists of the entire solutions (Lemma \ref{l:properties}), we have that for any two $u_0,v_0 \in \mathcal{E}$, $z(u_0,v_0)$ is finite. We characterize non-degeneracy in Subsection \ref{ss:characterize} and give further sufficient conditions for it to hold in Subsection \ref{ss:nondeg}. For example, we show in Example \ref{e:nondeg} that non-degeneracy of $\mathcal{E}$ holds for non-linearities $g=-\partial V(x,u)/\partial u$, with $V\in C^2(\mathbb{R}^2)$ and bounded from below.
\end{remark}

We now outline the concept of the proof of Theorems \ref{t:main1} and \ref{t:main1b}. The main tool is the zero number lifted to the space of measures. The zero-number has been established as a tool to study dynamics of (\ref{r:main}) mainly due to Matano's work \cite{Matano:82} (see \cite{Polacik:02} and references therein for an overview). We say that a zero $u_0(x)-v_0(x)=0$ is multiple, if $(u_0)_x(x)-(v_0)_x(x)=0$. (We also say that $u_0$ and $v_0$ intersect transversally at $x$ if it is a simple, and non-transversally if it is a multiple zero.) In the bounded case, if $\mu_0$ is a Borel probability measure on $\mathcal{X}^{\alpha}$, we define the zero function of $\mu_0$ as
\begin{equation}
Z(\mu_0)=\int_{\mathcal{X}^{\alpha}}z(u_0,v_0)d\mu_0(u_0)d\mu_0(v_0). \label{d:zero}
\end{equation}
We will show that $Z$ on the space of Borel probability measures on $\mathcal{X}^{\alpha}$ has analogous properties to the zero-function $z$ on $\mathcal{X}^{\alpha}$ (\cite{Fiedler89,Polacik:02} and references therein): for any $t > 0$, $Z(\mu(t))$ is essentially finite\footnote{We can always adjust the "weights" in the ergodic decomposition to make it finite, see Lemma \ref{l:4_2}.} (where $\mu(t)$ is the evolution of $\mu(0)=\mu_0$ induced by (\ref{r:main}) on the space of measures); it is non-increasing; and if there is a multiple zero of $u_0-v_0$ for some $u_0,v_0$ in the support of $\mu(t)$, then $Z(\mu(t))$ is strictly decreasing at $t$ in the following sense: for all $\delta >0$, $Z(\mu(t+\delta))< Z(\mu(t-\delta))$.

Importantly, the same technique applies also in the extended case, if we consider {\it $S$-invariant measures}. First, we note that there are many $S$-invariant measures on $\mathcal{X}^{\alpha}$ which are not supported only on periodic functions: e.g. consider the Bernoulli measure on the space of bi-infinite sequences of $0,1$, and associate to each sequence a function $u$ by combining two arbitrary smooth profiles $u^0,u^1:[0,1] \rightarrow \mathbb{R}$, $u^0(0)=u^0(1)=u^1(0)=u^1(1)$, as in Example \ref{e:nondeg}.

We again define the zero function as in (\ref{d:zero}), i.e. by considering only zeroes in $[0,1)$ (thus $Z(\mu_0)$ is typically finite). As the measure is $S$-invariant, it is the same as considering only zeroes in any $[y,y+1)$, $y \in \mathbb{R}$. The $Z(\mu_0)$ can be interpreted, and indeed for ergodic\footnote{This holds if $\mu_0 \times \mu_0$ is $S \times S$-ergodic; see Subsection \ref{ss:characterize}.} $\mu_0$ is the same for $\mu_0$-a.e. $u_0,v_0$ as the {\it average density of zeroes}
\begin{equation*}
\lim_{n \rightarrow \infty} \frac{1}{2n}\sum_{k=-n}^{n-1} z(S^ku_0,S^kv_0)< \infty, 
\end{equation*}
(this follows from the Birkhoff ergodic theorem and measurability of $z$ established in Lemma \ref{l:3_1}). Now, $Z(\mu(t))$ is non-increasing in $t$, as the flux of zeroes through $x=0$ and $x=1$ by the $S$-invariance of the measure cancels out. Finally, it may be somewhat counter-intuitive that a single multiple zero for some $x \in [0,1)$ causes the entire density of zeroes on the infinite line to decrease. The rationale for this is that by the local structure of zeroes (Lemma \ref{l:local}), a multiple zero of $u(t)-v(t)$ persists in an open neighbourhood $U \times V$ of $(u(t),v(t))$ for some $\tilde{t}$ close to $t$. By Poincar\'{e} recurrence, if $u_0,v_0$ are in the support of a $S$-invariant measure, one can find a positive measure subset of $W \subset U \times V$ for which a positive density of $S \times S$-translates visit $W$, thus a single multiple zero implies existence of a set of positive measure with a positive density of multiple zeroes along the real line for times close to $t$. We make this ad-hoc argument rigorous by using standard ergodic-theoretical tools, combined with the well-established local and global structure of zeroes \cite{Angenent:88,Chen:98}.

Considering $S$-invariant measures and the ergodic attractor in the extended case is related to analysing asymptotics for $\mu_0$-a.e. initial condition with respect to any $S$-invariant measure $\mu_0$. This approach was already taken by Sinai \cite{Sinai:91} in his study of the forced viscous Burgers equation, as we discuss in Section \ref{s:burgers}. We establish in Proposition \ref{p:transversal} an example of a general result in this direction used later: for $S$-invariant $\mu_0$ and $\mu_0$-a.e. $u_0$, $\omega(u_0)$ consists of orbits which do not intersect non-transversally a given $S,T$-invariant solution $v_0$ (i.e. a spatially and temporally periodic orbit).

\subsection{Statements of the results: Burgers-like equations and uniqueness of invariant measures}\label{ss:unique}

The second part of the paper focuses on establishing sufficient conditions for uniqueness of an invariant measure and implications, or equivalently on proving the generalized versions of results of Sinai \cite{Sinai:96}, (i)-(iii), mentioned in the introduction, and established in Corollaries \ref{c:bounded}, \ref{c:asymptotics} and \ref{c:weak*} below. The main tools in the proof are Theorem \ref{t:main1} and the zero function on the space of probability measures.

We say that an equation is Burgers-like, if the following holds: 

\begin{itemize}
	\item[(B1)] {\it Sub-quadratic growth of non-linearity in $u_x$:} There exists an $\varepsilon >0$ and a continuous function $c : \mathbb{R}^+ \rightarrow \mathbb{R}^+$ such that
	\begin{align*}
	&|f(t,x,u,\xi)| \leq c(\rho)\left(1+|\xi|^{2 - \varepsilon} \right) \\
	& \hspace{5ex} (\rho > 0, \: (t,x,u,\xi) \in [0,1]\times [0,1] \times [-\rho,\rho] \times \mathbb{R}  ).
	\end{align*}
	
	\item[(B2)] {\it Weak dissipation:} There exists an upper semi-continuous function $l : \mathbb{R} \rightarrow \mathbb{R}^+$ such that: if $u_0 \in H^{2 \alpha}(\mathbb{S}^1)$, $\int_0^1 u_0(x)dx = y$ and $||u_0-y||_{L^{\infty}(\mathbb{S}^1)} \leq l(y)$; and if the solution of (\ref{r:main}), $u(t_0)=u_0$ exists on $(t_0,t_1)$ for some $t_1 >t_0$, then for every $t \in [t_0,t_1)$ we have $||u(t)-y||_{L^{\infty}(\mathbb{S}^1)} \leq l(y)$. Furthermore, the function $l$ satisfies
	\begin{equation}
	\lim_{y \rightarrow \infty}(y-l(y))=+\infty, \quad \quad \lim_{y \rightarrow -\infty}(y+l(y))=-\infty. \label{r:B2}
	\end{equation}
	
	\item[(B3)] {\it Invariance:} For every $u_0 \in H^{2 \alpha}(\mathbb{S}^1)$, if the solution of (\ref{r:main}), $u(t_0)=u_0$ exists on $(t_0,t_1)$ for some $t_1 >t_0$, then for every $t \in [t_0,t_1)$, we have that $\int_0^1 u(x,t)dx = \int_0^1 u(x,t_0)dx$.
\end{itemize}

Recall the ordering on $\mathcal{X}^{\alpha}$: we write $u_0 \leq v_0$ if $u_0(x)\leq v_0(x)$ for all $x \in \mathbb{S}^1$, resp. $x \in \mathbb{R}$; $u_0 \ll v_0$ if $u_0(x) < v_0(x)$ for all $x \in \mathbb{S}^1$, resp. $x \in \mathbb{R}$; and $u_0 < v_0$ if $u_0 \leq v_0$ but $u_0 \neq v_0$. A family in $\mathcal{X}^{\alpha}$ is strongly totally ordered, if for all $u_0,v_0$ in the family, we have either $u_0 = v_0$ or $u_0 \ll v_0$.

We state results only for the more general time-periodic case; modifications for the autonomous case are straightforward and commented on throughout the paper.

\begin{thm} \label{t:main2} Assume (A1-3) and (B1-3) in the time-periodic case. 
	
	(i) There exists a set ${\mathcal{V}}=\lbrace v^y_0, \: y \in \mathbb{R} \rbrace$, $v^y_0 \in H^{2 \alpha}(\mathbb{S}^1)$, satisfying that $y \mapsto v^y_0$ is continuous as a map $\mathbb{R} \rightarrow H^{2 \alpha}(\mathbb{S}^1)$, strictly increasing, and such that for all $y \in \mathbb{R}$, $v^y_0$ is $T$-invariant and $\int_0^1 v^y_0(x)dx=y$. Furthermore, it is a unique family with these properties.
	
	(ii) In the bounded case, and in the extended case if $\mathcal{E}$ is non-degenerate, we have that $\mathcal{E}=\mathcal{V}$.
	
	(iii) In the bounded case, for all $y \in \mathbb{R}$ there is a unique invariant measure on $\mathcal{B}_y:=\lbrace u_0 \in \mathcal{X}^{\alpha}, \int_0^1 u_0(x)dx=y\rbrace$, concentrated on a single $v^y_0\in \mathcal{B}_y$.
\end{thm}

\begin{remark} 
	Note that in the time-periodic case, if $u_0$ is $T$-invariant, the solution $u(t)$ of (\ref{r:main}), $u(0)=u_0$ is not necessarily constant. It is $1$-periodic, i.e. has the same temporal periodicity as the nonlinearity.  
\end{remark}

We can now recover the conclusion (i) by Sinai on asymptotics of the Burgers equation in the bounded case, by applying general techniques of the order-preserving dynamics (in particular the Nonorderedness principle valid in the bounded case due to Hirsch \cite{Hirsch:88}; see also \cite{Polacik:02}, Section 3):
\begin{corollary} \label{c:bounded}
	Assume (A1-3) and (B1-3) in the B/P case. Then for each $u_0 \in \mathcal{X}^{\alpha}$, $\omega(u_0)=\lbrace v^{y_0}_0\rbrace$, where $y_0=\int_0^1 u_0(x)dx$ and $v^y_0$ is as in Theorem \ref{t:main2}, (i).
\end{corollary}

Let $\mathcal{V}$ be as in Theorem \ref{t:main2}, (i). To establish conclusions in the extended case, we again require a technical condition of finite density of zeroes:

\begin{itemize}
	\item[(N1)] Assume in the extended case that $\mu_0$ is a $S$-invariant Borel probability measure on $\mathcal{X}^{\alpha}$, supported on a set bounded in $\mathcal{X}^{\alpha}$, such that for every $v_0 \in \mathcal{V}$, and $\mu_0$-a.e. $u_0$, (\ref{d:density}) holds.
\end{itemize}

We give examples of many non-trivial measures satisfying (N1) without any a-priori knowledge of $\mathcal{V}$ in Remark \ref{e:nondeg}.

We denote by $\mathcal{H}$ the family (possibly empty) of all {\it spatially heteroclinic solutions} associated to $\mathcal{V}$, i.e. such that for $h_0 \in \mathcal{H}$, the solution of (\ref{r:main}), $h(0)=h_0$ exists for all $t \in \mathbb{R}$, such that $h_0$ intersect each $v^y_0 \in \mathcal{V}$ at most once, transversally, and such that for some $y_1 \neq y_2$, and for all $t \in \mathbb{R}$, $\lim_{x \rightarrow -\infty}|h_0(x)-v^{y_1}(x)|=0$, $\lim_{x \rightarrow \infty}|h_0(x)-v^{y_2}(x)|=0$. Note that by continuity of solutions with respect to initial conditions, we then have for all $t \in \mathbb{R}$:
\begin{equation}
\lim_{x \rightarrow -\infty}|h(x,t)-v^{y_1}(x,t)|=0, \hspace{5ex} \lim_{x \rightarrow \infty}|h(x,t)-v^{y_2}(x,t)|=0.
\end{equation}

We will establish the following:

\begin{corollary} \label{c:asymptotics}
	Assume (A1-3), (B1-3) in the E/P case, and let $\mu_0$ satisfy (N1).
	
	(i) For $\mu_0$-a.e. $u_0$, we have that $\omega(u_0)\subset \mathcal{V} \cup \mathcal{H}$.
	
	(ii) $\omega$-limit set of evolution of $\mu_0$ in the weak$^*$-topology consists of measures supported on $\mathcal{V}$.
\end{corollary}

To fully recover Sinai's conclusions (ii), (iii) in the extended case, we also require an additional control of the "oscillations" around the quantity conserved in the bounded case:

\begin{itemize}
	\item[(N2)] Assume in the extended case that $\mu_0$ is a $S$-invariant Borel-probability measure on $\mathcal{X}^{\alpha}$. Let $y_0 = \int \int_0^1u_0(x) \, dx \, d\mu_0(u_0)$, and assume that there exists $C>0$ and that for $\mu_0$-a.e. $u_0$, and all $t \geq 0$, $x \in \mathbb{R}$, 
	$$
	\left| \int_0^x u(z,t)dz - x\, y_0 \right| \leq C,
	$$
	where $u(t)$ is a solution of $(\ref{r:main})$, $u(0)=u_0$.
\end{itemize}

\begin{corollary} \label{c:weak*}
	Assume (A1-3), (B1-3) in the E/P case, and let $\mu_0$ satisfy (N1), (N2). Then:
	
	(ii) For $\mu_0$-a.e. $u_0$, we have $\omega(u_0)=\lbrace v^{y_0} \rbrace$.
	
	(ii) The $\omega$-limit set of $\mu_0$ is $\delta_{v^{y_0}}$, i.e. the Dirac measure concentrated on $v^{y_0} \in \mathcal{V}$.
\end{corollary}	

In Sections \ref{s:burgers} and \ref{s:short} we will give more general conditions (C1) and (C2) which can replace (B2-3) and (N2) respectively. For example, we will show that existence of a 1-dimensional, ordered family $\mathcal{V}$ as in Theorem \ref{t:main2},(i) suffices instead of the invariance property (B3). In Section \ref{burgers} we will see that the Burgers equation and a generalized form of it satisfy (A1-3), (B1-3); and that (N2), and in a certain sense (N1), were also originally assumed by Sinai \cite{Sinai:91}, thus our results are indeed a generalization of the aforementioned results for the Burgers equation.

\subsection{The structure of the paper} The paper is structured as follows: in Section \ref{s:prelim} we give the required background on existence of solutions of (\ref{r:main}), the choice of topologies (some of the technical definitions and remarks are moved to Appendix \ref{s:fractional}) and notation. We then in Sections \ref{s:zeros} and \ref{s:balance} recall the key properties of the zero number as the key tool, introduce the balance law of zeroes, and other key properties of the zero flux and the zero dissipation. We complete the second part of the paper by proving ergodic Poincar\'{e}-Bendixon theorems in Sections \ref{s:bounded} and \ref{s:extended}. In Sections \ref{s:notrans}-\ref{s:short} we prove results for Burgers-like equations in four logical steps divided into sections. In the fourth part of the paper, we show that the results apply to a family of generalized Burgers equation, then apply our theorems to other general and specific examples, and finally list some open problems. In Appendix \ref{s:fractional} we comment on function spaces in the extended case, and in Appendix \ref{s:ergodic} give interpretations of the ergodic attractor.

\begin{remark}
	All the results also hold for the equations $u_{t}=\varepsilon u_{xx}+g(t,x,u,u_{x})$, $\varepsilon >0$.
\end{remark}

\begin{remark}
	Theorems \ref{t:main1} and \ref{t:main1b} were already announced in \cite{Slijepcevic13b}, with derived further implications to the topological entropy of (\ref{r:main}) in all four cases considered here. All the results in \cite{Slijepcevic13b} in the extended case hold under an additional assumption of non-degeneracy of $\mathcal{E}$.
\end{remark}

\section{The function spaces and notation} \label{s:prelim}

In the autonomous case, (\ref{r:main}) with the assumptions (A1-4) generates a continuous semiflow on $\mathcal{X}^{\alpha}$, $3/4 < \alpha < 1$ (see e.g. \cite{Henry,Lunardi:95,Polacik:02} in the bounded case, and \cite{Gallay:01} in the extended case), denoted by $T(t)$, $t \geq 0$. In the extended case, we consider dynamics with respect to the continuous time-one map $T : \mathcal{X}^{\alpha} \rightarrow \mathcal{X}^{\alpha}$. We use the graph norm on $\mathcal{X}^{\alpha}$, $0 < \alpha < 1$:
$$
||u||_{\mathcal{X}^{\alpha}}:= ||A_1^{\alpha}u||_{\mathcal{X}},
$$
where $Au=-u_{xx}$ is the linear operator on $\mathcal{X}$ with the domain $D(A)=H^2(\mathbb{S}^1)$, resp. $D(A)=H^2_{\text{ul}}(\mathbb{R})$, with $A_1 = A + I$, and $A_1^{\alpha}$ is the standard fractional power (see \cite[Section 1.4]{Henry} and also Appendix \ref{s:fractional} for the extended case).

Let $\mathcal{B}$ be as in (A4). In the extended case, we need to equip $\mathcal{B}$ with a coarser topology to ensure compactness of invariant sets (see e.g. \cite{Gallay:01,Gallay:12} for a detailed discussion). We choose the topology of locally uniform convergence; however many choices induce equivalent topology on closed, bounded invariant sets (see Lemma \ref{l:topology} and the related discussion in Appendix \ref{s:fractional}). The semiflow $T(t)$, resp. the map $T$ are continuous also in the coarser topology.

Fix throughout the paper a (small) $\delta_0< 0$. Let $\tilde{ \mathcal{B}}$ be the closure of the set of all $u_0 \in \mathcal{B}$ for which the solution of (\ref{r:main}), $u(0)=u_0$ exists backwards in time on $\mathcal{B}$ for $t \in (\delta_0,0)$ (not necessarily uniquely). Then  $\tilde{ \mathcal{B}}$ is invariant and compact. In the bounded case, this follows from the compact embedding of $X^{\gamma}$ in $X^{\alpha}$ and the variation of constants formula. Similarly this can be established in the extended case (see Lemma \ref{l:topology}, also \cite{Gallay:01} for further discussion). As all the trajectories are eventually in $\tilde{\mathcal{B}}$, it suffices to consider the dynamics and invariant measures on $\tilde{\mathcal{B}}$.

Whenever required for clarity, we denote elements of $\mathcal{B}$ and $\tilde{\mathcal{B}}$ with indices as $u_0,v_0,$... and by $u(t),v(t),...$ the solutions of (\ref{r:main}) with the initial conditions $u_0,v_0,$... respectively. Let $\mathcal{M}(\mathcal{B})$ be the set of all the invariant Borel probability measures on $\mathcal{B}$ (invariant with respect to actions in Table \ref{t:table}). Analogously we denote the measures on $\mathcal{B}$ with indices as $\mu_0,\nu_0$,..., and by $\mu(t),\nu(t)$,... their evolution, i.e. pushed $\mu(0)=\mu_0$ with respect to the time-$t$ map generated by (\ref{r:main}).

The proofs require considering dynamics of two replicas of (\ref{r:main}), i.e. a dynamical system on $\tilde{\mathcal{X}}^2$. We use $\hat{.}$ to denote certain functions on $\tilde{\mathcal{X}}^2$, e.g. $\hat{S}=S \times S$, $\hat{T}=T \times T$.

We frequently use the fact that (\ref{r:main}) is strongly monotone, i.e. that if $u(t_0) < v(t_0)$, then for all $t \geq t_0$ for which both solutions exist, $u(t) \gg v(t)$.

Assuming (A1-4), the equation (\ref{r:main}) considered on $\mathcal{B}$ admits an attractor $\mathcal{A}$ (\cite{Raugel:02}, Section 2.3), which is unique, compact, and characterized as the set of all $u_0 \in \mathcal{B}$ such that the solution of (\ref{r:main}), $u(0)=u_0$ exists for all $t \in \mathbb{R}$. The attractor $\mathcal{A}$ depends on the choice of $\mathcal{B}$, thus we write $\mathcal{A}(\mathcal{B})$ when the choice of a set $\mathcal{B}$ satisfying (A4) is not clear from the context. 

Finally, we note the properties essential for considerations involving the zero number. 

\begin{lemma} \label{l:c1} Fix $t_0 \in \mathbb{R}$.
	
	(i) For any $t >t_0$, $x,y \in \mathbb{R}$, $x < y$, the mapping $\mathcal{\tilde{B}} \mapsto C^1([x,y])$ defined with $u(t_0) \mapsto u(.,t)|_{[x,y]}$ is continuous.
	
	(ii) For any $t > s > t_0$, $x \in \mathbb{R}$, the mapping $\mathcal{\tilde{B}} \mapsto C^1([s,t])$ defined with $u(t_0) \mapsto u(x,.)|_{[s,t]}$ is continuous.
\end{lemma}

\begin{proof}
	The claim (i), as well as continuity of $u(t_0) \mapsto u(x,.)|_{[s,t]}$ as $\tilde{\mathcal{B}} \rightarrow C^0([s,t])$, follows from continuous embedding of $H^{2\alpha}(\mathbb{S}^1)$ in $C^1(\mathbb{S})$, respectively $H^{2\alpha}_{\text{loc}}(\mathbb{R})$ in $C^1_{\text{loc}}(\mathbb{R})$, continuous dependence on initial conditions in $\tilde{\mathcal{B}}$, and continuity of $t \mapsto u(t)$ in $\tilde{B}$ for the latter claim.  To complete (ii), it suffices to show continuity of $u(t_0) \mapsto u_t(x,.)$ as $\tilde{\mathcal{B}} \rightarrow C^0([s,t])$. This follows from e.g. \cite{Henry}, Theorem 3.5.3, with the choice of the spaces as in the proof of local existence of solutions (in the extended case, we in addition apply continuous dependence on initial conditions in $\tilde{\mathcal{B}}$). 
\end{proof}

\begin{remark}
	Note that we do not assume strong dissipativity conditions on $g$, such as e.g. (G1-3) in \cite{Polacik:02}, as they would not cover the Burgers-like equations considered in the second part of the paper. 
\end{remark}
	
\begin{remark}
	For an argument alternative to Lemma \ref{l:c1} enabling applying zero-number techniques for even less smooth $g$ than those satisfying (A1-3), refer to \cite{Polacik:15}, Section 2.
\end{remark}

\part{Ergodic Poincar\'e-Bendixson theorems}

\section{Preliminaries on the set of zeroes} \label{s:zeros}

In this section we consider properties of the set of zeroes of $u(t)-v(t)$, where $u,v$ are two solutions of (\ref{r:main}) on $\tilde{\mathcal{B}}$. In addition to the zero function $Z_w$ associated to the curve $w(t)=u(t)-v(t)$, we introduce the notions of the flux of zeroes $F_w$ and the dissipation of zeroes $D_w$, analogous to the notions of energy flux and energy dissipation considered e.g. in \cite{Gallay:01, Gallay:12}. The main results of the section are the balance law for the flux of zeroes (\ref{r:balance}), and sufficient conditions for continuity of $Z_w, F_w, D_w$. The proofs rely on the well-known local and global structure of the set of zeroes, which we recall first.

In this section we assume (A1-4) and fix $u_0,v_0 \in \tilde{\mathcal{B}}$ for which the solutions $u(t),v(t)$ of (\ref{r:main}), $u(0)=u_0$, resp. $v(0)=v_0$ exist on $(\delta_0,\infty)$, where $\delta_0 <0 $ is as in Section \ref{s:prelim}. Denote by $w_0=u_0-v_0$ and $w(t)=u(t)-v(t)$, $t \in (\delta_0,\infty)$. Let $N_w$ be the set of zeroes (or the nodal set), and $S_w$ the set of multiple (or singular) zeroes associated to $w \neq 0$, defined with
\begin{align*}
N_w & := \lbrace (x,t) \in \mathbb{R} \times (\delta_0,\infty)  \: : \: w(x,t) = 0 \rbrace,  \\
S_w & := \lbrace (x,t) \in \mathbb{R} \times (\delta_0,\infty) \: : \: w(x,t) = w_x(x,t) = 0 \rbrace. 
\end{align*}
For $u=v$, i.e. $w=0$, we set $S_w=N_w=\emptyset$.

The following local and global structure of zeroes is well-known, proved by Chen \cite{Chen:98} (for earlier, less complete description by Angenent and Chen and Pol\'{a}\v{c}ik see \cite{Angenent:88,Chen:96}): 
 
\begin{lemma} \label{l:local}
	{\bf Local structure of zeroes.} If $(x_0,t_0) \in N_w$, then there is a neighbourhood $Q = [x_0-\varepsilon, x_0+\varepsilon] \times [t_0 - \delta, t_0 + \delta]$, $\varepsilon, \delta >0$ of $(x_0,t_0)$ such that the following properties hold:
	
\begin{itemize}
	\item[(a)] If $(x_0,t_0) \notin S_w$, then $Q \cap N_w$ equals a single curve $\lbrace (\gamma(t),t): t \in [t_0 - \delta, t_0 + \delta] \rbrace$, where $\gamma: [t_0 - \delta, t_0 + \delta] \rightarrow \mathbb{R}$ is of class $C^1$ and $\gamma(t_0)=x_0$.
	
	\item[(b)] If $(x_0,t_0) \in S_w$, then there is an integer $m \geq 2$ (the degree of the zero) such that the following holds: 
	
	\begin{itemize} 
		\item[(b1)] For even $m$, there exist $m$ curves $\gamma_1,...,\gamma_m : [t_0 - \delta, t_0) \rightarrow \mathbb{R}$ of class $C^1$, such that
	\begin{equation}
	 \gamma_1(t) < \gamma_2(t) < ... < \gamma_m(t) \hspace{5ex} \text{for all } t \in [t_0 - \delta, t_0), \label{r:small}
	\end{equation}
	such that $\lim_{t \rightarrow t_0^-}\gamma_k(t) = x_0$, $k=1,...,m$ and such that $Q \cap N_w$ equals union of $\lbrace (\gamma_j(t),t): t \in [t_0 - \delta, t_0 ) \rbrace$, $j=1,...,m$, and $\lbrace (x_0,t_0) \rbrace$.
	
	\item[(b2)] For odd $m$, there exist $m$ curves $\gamma_1,...,\gamma_{(m-1)/2},\gamma_{(m+3)/2},..., \gamma_m : [t_0 - \delta, t_0) \rightarrow \mathbb{R}$, $\gamma_{(m+1)/2} : [t_0 - \delta, t_0+\delta ] \rightarrow \mathbb{R}$ of class $C^1$, satisfying (\ref{r:small}), such that $\lim_{t \rightarrow t_0^-}\gamma_j(t) = x_0$, $j=1,...,(m-1)/2,(m+3)/2,...,m$, such that $\gamma_{(m+1)/2}(t_0)=x_0$, and such that $Q  \cap N_w$ equals union of $\lbrace (\gamma_j(t),t): t \in [t_0 - \delta, t_0 ) \rbrace$, $j=1,...,(m-1)/2,(m+3)/2,...,m$ and  $\lbrace (\gamma_{(m+1)/2)}(t),t): t \in [t_0 - \delta, t_0 + \delta ] \rbrace$.
	\end{itemize}
	In both cases, $\lbrace (x_0,t_0) \rbrace$ is equal to $Q \cap S_w $.
\end{itemize}
\end{lemma}
From this we can deduce the following global structure of the set of zeroes.

\begin{lemma} \label{l:global} {\bf Global structure of zeroes.} 
   There exist an at most countable family of curves $\gamma_i : (\delta_0,d_i) \rightarrow \mathbb{R}$ of class $C^1$ associated to $w$, $d_i \in (\delta_0,\infty]$, $i \in {\mathcal I}_w$, ${\mathcal I}_w$ a finite set or $\mathbb{N}$, satisfying the following:
   \begin{itemize}
   	\item[(i)]  The sets $\lbrace (\gamma_i(t),t), \: t \in (\delta_0,d_i) \rbrace$, $i \in \mathcal{I}$, are disjoint.
   	
   	\item[(ii)] $S_w = \cup_{i \in \mathcal {I}_w, d_i< \infty} \lbrace (\lim_{t \rightarrow d_i^-} \gamma_i(t),d_i) \rbrace$.
   	
   	\item[(iii)] $N_w = \cup_{i \in \mathcal {I}_w} \lbrace (\gamma_i(t),t), \: t \in (\delta_0,d_i) \rbrace \cup S_w$.
   	
   	\item[(iv)] For each compact $Q \subset \mathbb{R}^2$, there exists at most finitely many $i \in \mathcal{I}$ such that $\lbrace (\gamma_i(t),t), t \in (\delta_0,d_i) \rbrace $ intersects $Q$. Specifically, there are at most finitely many multiple zeroes in $Q$.
   \end{itemize}
\end{lemma}

For the proof, see the proof of Lemma 2.3 in \cite{Feireisl:00}, taking into account adjustments of the statement fitting our purposes (see Remark \ref{rm:fix} below).

For simplicity of notation, we drop the dependency on $w$ in the notation of curves of zeroes $\gamma$. For $i \in \mathcal{I}_w$ such that $d_i < \infty$, denote by $x_i = \lim_{t \rightarrow d_i^-}\gamma_i(t)$, and then $S_w = \lbrace (x_i,d_i), d_i < \infty, \: i \in \mathcal{I}_w \rbrace$. For $d_i < \infty$, let $\overline{\gamma}_i : (-\infty, d_i] \rightarrow \mathbb{R}$ be the unique continuous extension of $\gamma_i$ (i.e. such that $\overline{\gamma}_i(d_i)=x_i$), and for $d_i=\infty$ let $\overline{\gamma}_i=\gamma_i$.


We define the number of zeroes $Z_w$ in $[x,y) \times \lbrace t \rbrace$, the flux $F_w$ of zeroes through $\lbrace x \rbrace \times [s,t)$, and the dissipation $D_w$ of zeroes in $[x,y)\times (s,t]$, defined for $\delta_0< s < t$, $x< y$, $s,t,x,y \in \mathbb{R}$, associated to $w \neq 0$, as follows. Let $Q \subset \mathbb{R} \times (\delta_0,\infty)$ be any compact set containing $[x, y] \times [s,t]$, and let $\mathcal{I}_w(Q)$ be the set of all $i \in \mathcal{I}_w$ such that $\lbrace (t,\gamma_i(t)), t \in ((\delta_0,d_i))\rbrace$ intersects $Q$. Now we define
\begin{align} \label{d:ZFD}
\begin{split}
 Z_w(x,y,s)& = \left| i \in \mathcal{I}_w(Q), \, d_i>s, \, \gamma_i(s) \geq y \right| - \left| i \in \mathcal{I}_w(Q),  d_i>s,\, \gamma_i(s) \geq x \right|,
   \\
 F_w(x,s,t) & = \left| i \in \mathcal{I}_w(Q), \,  d_i>s, \, \overline{\gamma}_i(\min(t,d_i)) \geq x \right| - \left| i \in \mathcal{I}_w(Q), \, d_i>s,\, \gamma_i(s) \geq x \right|, \\
 D_w(x,y,s,t) & = \left| i \in \mathcal{I}_w(Q), \, (x_i,d_i) \in [x,y) \times (s,t] \right|,
\end{split}
\end{align}
where $|.|$ denotes the cardinal number of a set, always finite by Lemma \ref{l:global}, (iv). Also it is easy to verify that the definition above does not depend on the choice of $Q$. 
For $w=0$, we set $Z_w=F_w=D_w=0$ independently of the arguments. Note that the function $D_w$ counts multiple zeroes in $[x,y) \times [s,t)$ with their multiplicity ($m$ times for even, $m-1$ times for odd $m$).

\begin{remark} \label{rm:fix}
	For technical reasons, our definition of the curves of zeroes $\gamma_i$ slightly differs from e.g. \cite{Feireisl:00,Fiedler89}, as the even, multiple zeroes are not in the union of images $(t,\gamma_i(t))$. Also the zero function $Z_w$, does not "count" even, multiple zeroes. This simplifies definitions of the flux and dissipation of zeroes, as the images of $\gamma_i$ are disjoint. Note that all the multiple zeroes are properly "counted" by the dissipation function $D_w$.
\end{remark}

We now obtain the following balance law:
\begin{lemma}  \label{l:3_2} 
	{\bf The balance law for the flux of zeroes.} Let $x,y,s,t \in \mathbb{R}$ such that $0 \leq s < t$. Then
	\begin{align}
	   Z_w(x,y,t)-Z_w(x,y,s) = F_w(y,s,t)-F_w(x,s,t) - D_w(x,y,s,t). \label{r:balance}
	\end{align}

\end{lemma}

\begin{proof} If $w=0$, the claim is trivial. For $w \neq 0$, let $Q$ be as in (\ref{d:ZFD}).
	First note that
	\begin{align*}
	\left| i \in \mathcal{I}_w(Q), \,  d_i>s, \, \overline{\gamma}_i(\min(t,d_i)) \geq x \right| & = 
	\left| i \in \mathcal{I}_w(Q), \, d_i>t, \, \gamma_i(t) \geq x \right| \\ & \hspace{3ex} +	\left| i \in \mathcal{I}_w(Q), \, (x_i,d_i) \in [x,\infty) \times (s,t]	 \right|.
	\end{align*}
	It suffices now to insert that and the definition (\ref{d:ZFD}) in the left, resp. right-hand side of (\ref{r:balance}) to verify it.
\end{proof}

We now establish certain continuity properties of $Z_w, F_w, S_w$. In the following lemma, $\tilde{u}$, $\tilde{v}$ denote solutions of (\ref{r:main}), $\tilde{u}(0)=\tilde{u}_0$, resp. $\tilde{v}(0)=\tilde{v}_0$.

\begin{lemma} \label{l:continuity} 
	{\bf Continuity of zero functions.} Let $x,y,s,t \in \mathbb{R}$ such that $0 \leq s < t$. Then
	\begin{itemize}
		\item[(i)] If all zeroes in $[x,y)\times \lbrace t \rbrace $ are regular, and there are no zeroes in $\lbrace (x,t),(y,t) \rbrace$, then there is an open neighbourhood $\mathcal{U}$ of $(u_0,v_0)$ in $\tilde{\mathcal{B}}^2$ such that for all $\tilde{u}_0,\tilde{v}_0 \in \mathcal{U}$, $\tilde{w}=\tilde{u}-\tilde{v}$, we have $Z_w(x,y,t)=Z_{\tilde{w}}(x,y,t)$.
		
		\item[(ii)] If all zeroes in $\lbrace x \rbrace \times [s,t)$ are regular, and there are no zeroes in $\lbrace (x,s),(x,t) \rbrace$, then there is an open neighbourhood $\mathcal{U}$ of $(u_0,v_0)$ in $\tilde{\mathcal{B}}^2$ such that for all $\tilde{u}_0,\tilde{v}_0 \in \mathcal{U}$, $\tilde{w}=\tilde{u}-\tilde{v}$, $F_w(x,s,t)=F_{\tilde{w}}(x,s,t)$.
		
		\item[(iii)] Assume all the zeroes in $\partial Q$, where $Q = [x,y] \times [s,t]$, are regular. Then there is an open neighbourhood $\mathcal{U}$ of $(u_0,v_0)$ in $\tilde{\mathcal{B}}^2$ such that for all $\tilde{u}_0,\tilde{v}_0 \in \mathcal{U}$, $\tilde{w}=\tilde{u}-\tilde{v}$, $D_w(x,y,s,t)=D_{\tilde{w}}(x,y,s,t)$.	
	\end{itemize}
\end{lemma}

\begin{proof} Throughout the proof, we implicitly use Lemma \ref{l:c1} several times.	
By embedding of $\mathcal{X}^{\alpha}$ in $C^1(\mathbb{R})$ and continuous dependence on initial conditions of (\ref{r:main}), we can find an open neighbourhood of $\mathcal{U}$ of $(u_0,v_0)$ such that for each $(\tilde{u}_0,\tilde{v}_0) \in \mathcal{U}$, for $\tilde{w}=\tilde{u}-\tilde{v}$ the assumptions on zeroes in (i) hold for $\tilde{w}$. Now $Z_{\tilde{w}}$ can be expressed with $Z_{\tilde{w}}(x,y,t)=|\tilde{w}(.,t)^{-1}(0)\cap (x,y)|$, and all the zeroes of $\tilde{w}$ in $(x,y) \times \lbrace t \rbrace$ are regular, i.e. whenever $\tilde{w}(z,t)=0$, $z \in (x,y)$, we have $w_x(z,t)\neq 0$. It suffices now to cover $(x,y)$ with finitely many sufficiently small open intervals and apply the implicit function theorem to $z \rightarrow w(z,t)$ on the open intervals containing a zero of $w$ to deduce (i).

To prove (ii), consider an open neighbourhood $\mathcal{U}_1$ of $(u_0,v_0)$, such that for some $\delta_1>0$ small enough, $(\tilde{u}_0,\tilde{v}_0) \in \mathcal{U}_1$, for $\tilde{w}=\tilde{u}-\tilde{v}$ we have 
\begin{equation} \label{r:allzero}
Z_{\tilde{w}}(x-\delta_1,x,s)=0, \hspace{3ex} Z_{\tilde{w}}(x-\delta_1,x,t)=0, \hspace{3ex} D_{\tilde{w}}(x-\delta_1,x,s,t)=0
\end{equation}
 (such an $\mathcal{U}_1$ exists by the assumptions on zeroes in (ii)). Now for all $0<\delta\leq \delta_1$, by the definition of $Z_w,D_w$, the assumptions on zeroes in (ii) hold also for $(x-\delta)$ instead of $x$. We claim that we can find $0 < \delta_2 \leq \delta_1$ such that, in addition, for all the zeroes in $\lbrace x-\delta_2 \rbrace \times [s,t)$ expressed as $(\tau,\gamma_i(\tau))$, $\gamma_i(\tau)=x-\delta_2$, $s \leq \tau < t$, we have that $(\gamma_i)_x(\tau)\neq 0$. We deduce that by applying the Morse-Sard Lemma to every $C^1$ function $\gamma_i$ to establish that the set of critical values $x \in \mathbb{R}$, i.e. $x$ for which $\gamma_i(\tau)=x$ and $(\gamma_i)_x(\tau)= 0$ for some $\tau \in [s,t]$, has the Lebesgue measure 0. As there are at most countably many curves of zeroes $\gamma_i$, $i \in \mathcal{I}_w$, this completes the proof of existence of such $\delta_2$. 

It is easy to verify that now
	\begin{align}
	F_w(x-\delta_2,s,t) & = \sum_{\tau \in (s,t), w(x-\delta_2,\tau) = 0} - \operatorname{sgn} (w_t(x-\delta_2,\tau)), \label{f:alternative}
	\end{align}
	as $\operatorname{sgn} (\gamma_i)_x(\tau)=-\operatorname{sgn} w_t(\gamma_i(\tau),\tau)$ whenever $(\gamma_i)_x(\tau)\neq 0$. We now find an open neighbourhood $\mathcal{U} \subset \mathcal{U}_1$ of $(u_0,v_0)$ such that for each $(\tilde{u}_0,\tilde{v}_0)\in \mathcal{U}$, $\tilde{w}=\tilde{u}-\tilde{v}$, for all the zeroes of $\tilde{w}$ in $\lbrace x-\delta_2 \rbrace \times (s,t)$, we have that $\tilde{w}_t(x-\delta_2,\tau) \neq 0$. Applying the implicit function theorem analogously as when proving (i), but now for the function $\tau \mapsto w(x-\delta_2,\tau)$, $\tau \in [s,t]$, we deduce that $F_{\tilde{w}}(x-\delta_2,s,t)$ is constant on $\mathcal{U}$. By (\ref{r:allzero}), the balance law (\ref{r:balance}) and the construction of $\mathcal{U}_1$ we see that $F_{\tilde{w}}(x,s,t)=F_{\tilde{w}}(x-\delta_2,s,t)$ on $\mathcal{U}$, which completes (ii).
	
	To show (iii), note first that for any $\delta_1,\delta_2,\varepsilon_1,\varepsilon_2 >0$ small enough, we have
	$$
	D_w(x,y,s,t)=D_w(x+\delta_1,y-\delta_1,s+\varepsilon_1,t-\varepsilon_1)=D_w(x-\delta_2,y+\delta_2,s-\varepsilon_2,t+\varepsilon_2)
	$$
	(this follows from the finiteness of the number of multiple zeroes in any compact $Q$ and the assumptions on zeroes in (iii)). In addition, by the local structure of zeroes we can choose $\delta_1,\delta_2,\varepsilon_1,\varepsilon_2 >0$ such that there are no zeroes in the "corners" $\lbrace x+\delta_1,y-\delta_1 \rbrace \times \lbrace s+\varepsilon_1,t-\varepsilon_1 \rbrace$ and $\lbrace x-\delta_2,y+\delta_2 \rbrace \times \lbrace s-\varepsilon_2,t+\varepsilon_2 \rbrace$. Now applying twice the balance law (\ref{r:balance}) and (i), (ii) (i.e. on $[ x+\delta_1,y-\delta_1] \times [s+\varepsilon_1,t-\varepsilon_1 ]$ and $[ x-\delta_2,y+\delta_2] \times [s-\varepsilon_2,t+\varepsilon_2]$), we can find a neighbourhood $\mathcal{U}$ of $(u_0,v_0)$ such that 
	$$
	D_w(x,y,s,t)=D_{\tilde{w}}(x+\delta_1,y-\delta_1,s+\varepsilon_1,t-\varepsilon_1)=D_{\tilde{w}}(x-\delta_2,y+\delta_2,s-\varepsilon_2,t+\varepsilon_2)
	$$
	for $(\tilde{u}_0,\tilde{v}_0) \in \mathcal{U}$, $\tilde{w}=\tilde{u}-\tilde{v}$. To establish $D_{\tilde{w}}(x,y,s,t)=D_w(x,y,s,t)$ on $\mathcal{U}$, it suffices to note that by the definition of $D_w$,
	$$
	D_{\tilde{w}}(x+\delta_1,y-\delta_1,s+\varepsilon_1,t-\varepsilon_1) \leq D_{\tilde{w}}(x,y,s,t) \leq D_{\tilde{w}}(x-\delta_2,y+\delta_2,s-\varepsilon_2,t+\varepsilon_2).
	$$
\end{proof}

\section{Properties of the zero, zero flux and zero dissipation functions} \label{s:balance}

We consider here the zero, zero flux and zero dissipation functions $z,f,d : \tilde{\mathcal{B}} \times \tilde{\mathcal{B}} \rightarrow \mathbb{R}$, defined as
\begin{equation}
z(u_0,v_0)=Z_w(0,1,0), \hspace{5ex} f(u_0,v_0)=F_w(0,0,1),  \hspace{5ex} d(u_0,v_0)=D_w(0,1,0,1), \label{d:zfd}
\end{equation}
where $w(t)=u(t)-v(t)$ and $u(t),v(t)$ are solutions of (\ref{r:main}) with the initial conditions $u(0)=u_0$, $v(0)=v_0$. In this section we first reformulate the balance law (\ref{r:balance}) in terms of $z,f,d$, then show that the property $d(u_0,v_0)>0$ persists in a certain sense for small perturbations in $\tilde{\mathcal{B}}^2$, and finally that $z,f,d$ are Borel-measurable.

\begin{remark} The zero function $z$ in the literature depends on one argument $w_0=u_0-v_0$. The flux and dissipation functions $f,d$, however, depend on both $u_0$, $v_0$ (and their evolution), so we adopt the same convention to $z$.	
\end{remark}

Note that the values of $z,f,d$ are always integers, and that $z,d \geq 0$.
Let $\hat{S},\hat{T}:\tilde{\mathcal{B}}^2 \rightarrow \tilde{\mathcal{B}}^2$, $\hat{S}(u_0,v_0)=(Su_0,Sv_0)$, $\hat{T}(u_0,v_0)=(Tu_0,Tv_0)$. By inserting $x=0$, $y=1$, $s=0$, $t=1$, the balance law of zeroes (\ref{r:balance}) can now be written as
\begin{equation}
z \circ \hat{T} - z = f \circ \hat{S}  - f - d. \label{r:bez!}
\end{equation}

\begin{lemma} \label{l:3_3}
	If $u_0,v_0 \in \tilde{\mathcal{B}}$ are such that $d(u_0,v_0)>0$, then there exists an open neighbourhood $\mathcal{U}$ of $(u_0,v_0)$ in $\tilde{\mathcal{B}}^2$ such that for each $(\tilde{u}_0,\tilde{v}_0) \in \mathcal{U}$, we have
	\begin{equation}
	(\tilde{u}_0,\tilde{v}_0)+d(\hat{S}^{-1}(\tilde{u}_0,\tilde{v}_0))+d(\hat{T}(\tilde{u}_0,\tilde{v}_0))+d(\hat{S}^{-1}\hat{T}(\tilde{u}_0,\tilde{v}_0)) \geq 1.
	\label{r:disseq}
	\end{equation}
\end{lemma}

\begin{proof} We denote below by $w=u-v, \tilde{w}=\tilde{u}-\tilde{v}$ the solution of (\ref{r:main}) with the initial conditions $u_0,v_0,\tilde{u}_0,\tilde{v}_0$ at $t=0$ respectively.
	By finiteness of the number zeroes in a compact set, we can find $0 < \delta, \varepsilon <1$ small enough such that for $Q=[-\delta,1-\delta] \times [\varepsilon,1+\varepsilon]$, there are no multiple zeroes in $\partial Q$, and such that $D_w(-\delta,1-\delta,\varepsilon,1+\varepsilon)=D_w(0,1,0,1)=d(u_0,v_0)$. Now we apply 
	Lemma \ref{l:continuity}, (iii), and find an open neighbourhood $\mathcal{U}$ of $(u_0,v_0)$ such that for each $(\tilde{u}_0,\tilde{v}_0) \in \mathcal{U}$, $\tilde{w}=\tilde{u}_0-\tilde{v}_0$, we have
	$D_w(-\delta,1-\delta,\varepsilon,1+\varepsilon)=D_{\tilde{w}}(-\delta,1-\delta,\varepsilon,1+\varepsilon)$.
	Finally it suffices to note that
\begin{align*}
d(\tilde{u}_0,\tilde{v}_0)+d(\hat{S}^{-1}(\tilde{u}_0,\tilde{v}_0))+d(\hat{T}(\tilde{u}_0,\tilde{v}_0))+d(\hat{S}^{-1}\hat{T}(\tilde{u}_0,\tilde{v}_0)) & =D_{\tilde{w}}(-1,1,0,2) \\ & \geq  D_{\tilde{w}}(-\delta,1-\delta,\varepsilon,1+\varepsilon) \\ & =d(u_0,v_0) \geq 1.
\end{align*}
\end{proof}

\begin{lemma} \label{l:3_1} 
	The functions $z,d,f : \tilde{\mathcal{B}}^2 \rightarrow \mathbb{R}$ are Borel-measurable. 
\end{lemma}

In order to prove it, we need the following:

\begin{lemma} \label{l:53little}
	Assume $u_0,v_0 \in \tilde{\mathcal{B}}$, let $u,v$ be the solutions of (\ref{r:main}) with the initial conditions $u_0,v_0$ at $t=0$, and let $w=u-v$. Then there exists $n_0(w)$ so that for all $n \geq n_0$, $x_n:=-1/\sqrt{n}$, $y_n:=1-1/\sqrt{n}$ and $t_n:=1/n$, 
	
(i) all the zeroes of $w$ in $[x_n,y_n) \times \lbrace t_n \rbrace$ are regular, there are no zeroes in $\lbrace (x_n,t_n),(y_n,t_n) \rbrace$, and

(ii) $z(w)=Z_w(x_n,y_n,t_n)$.
	
\end{lemma}

\begin{proof} Firstly, by finiteness of the number of multiple zeroes in a compact set, there are no multiple zeroes in $[x_n,y_n)\times \lbrace t_n \rbrace$ for $n$ large enough. Now if $(0,0)$ and $(1,0)$ are not zeroes, (i) holds for $n$ large enough by continuity of $w$. If $(0,0)$ and $(1,0)$ are even, multiple zeroes, then (i) holds for $n$ large enough by the local structure of zeroes Lemma \ref{l:local}.
	
Assume $(0,0)$ is a regular or odd multiple zero, by the local structure of zeroes it lies on a $C^1$ curve of zeroes $\gamma_i$, $\gamma_i(0)=0$, with a domain containing an open neighbourhood of $0$.
Now by the local structure of zeroes $(x_n,t_n)$ can be a zero for $n$ large enough only if $x_n=\gamma_i(t_n)$. However, by the choice of $t_n,x_n$, this is impossible for $n$ large enough, as $|(\gamma_i)_t(0)|$ would have to be $+\infty$ which contradicts it being $C^1$. Analogously $(y_n,t_n)$ can not be a zero for $n$ large enough, thus (i) holds.

Consier $Q, \mathcal{I}_w(Q)$ as in \ref{d:ZFD}. Analogously as above we deduce that for $n$ large enough, 
\begin{align*}
|i \in \mathcal{I}_w(Q), \, d_i > 0, \gamma_i(0) \geq 0| & = |i \in \mathcal{I}_w(Q), \, d_i > t_n, \gamma_i(t_n) \geq x_n|, \\
|i \in \mathcal{I}_w(Q), \, d_i > 0, \gamma_i(0) \geq 1| & = |i \in \mathcal{I}_w(Q), \, d_i > t_n, \gamma_i(t_n) \geq y_n|,
\end{align*}
which by definition gives (ii).
\end{proof}

\begin{proof}[Proof of Lemma \ref{l:3_1}]
	Denote by $\mathcal{U}_{n,k} \subset \tilde{\mathcal{B}}^2$ the set of all $(u_0,v_0)$ for which the following holds: let $w=u-v$, $u,v$ solutions of (\ref{r:main}) with the initial conditions $u_0,v_0$  at $t=0$ respectively, and let $t_n,x_n,y_n$ be as in Lemma \ref{l:53little}. Let $\mathcal{U}_{n,k}$ be the set of all $(u_0,v_0) \in \tilde{ \mathcal{B}}^2$ for which the statement in Lemma \ref{l:53little}, (i) holds, and for which $Z_w(x_n,y_n,t_n)=k$. We claim that $\mathcal{U}_{n,k}$ is open in $\tilde{ \mathcal{B}}^2$. Indeed, this follows from the continuous dependence of solutions of (\ref{r:main}) on initial conditions, from the local structure of zeroes and from Lemma \ref{l:continuity}, (i). Now, by Lemma \ref{l:53little},
	$$
	\lbrace (u_0,v_0)\in \tilde{ \mathcal{B}}^2, \: z(u_0,v_0)=k \rbrace = \bigcup_{m=1}^{\infty} \bigcap_{n=m}^{\infty}\mathcal{U}_{n,k},
	$$
	thus $z$ is measurable. The proof of measurability of $f$ is analogous. Measurability
	of $d$ follows from (\ref{r:bez!}).
\end{proof}

\section{The proof of Theorem \ref{t:main1} (the bounded case)} \label{s:bounded}

In this section we consider only the bounded case, and show the following:
\begin{proposition} \label{p:diszero1}
	For any $(u_0,v_0) \in \mathcal{E}^2$, we have that $d(u_0,v_0)=0$.
\end{proposition}

From this we will directly deduce Theorem \ref{t:main1} at the end of the section. We prove Proposition \ref{p:diszero1} in the time-periodic case only; the autonomous case is analogous (by taking the semiflow $T(t)$ instead of the map $T$). As discussed in the introduction, we define the zero function $\hat{Z}$ of two Borel probability measures on $\tilde{\mathcal{B}}$ as
$$
\hat{Z}(\mu_1,\mu_2):=\int_{\tilde{B}^2}z(u_0,v_0)d\mu_1(u_0)\mu_2(v_0),
$$
which is well-defined by non-negativity and Borel measurability of $z$. 

\begin{remark} \label{rr:61}
Analogously we can define $\hat{Z}$ for any Borel probability measure $\nu$ on $\tilde{\mathcal{B}}^2$ with $\hat{Z}(\nu_0)=\int_{\tilde{B}^2}z\, d\nu_0$. We slightly abuse the notation and write interchangeably $\hat{Z}(\mu_1 \times \mu_2)$ and $\hat{Z}(\mu_1,\mu_2)$.
\end{remark}
We can now write (\ref{d:zero}) as $Z(\mu_0)=\hat{Z}(\mu_0,\mu_0)$. The proposition \ref{p:diszero1} will follow from an analogous consideration as in the proof of Fiedler and Mallet-Paret Poincar\'{e}-Bendixson theorem \cite{Fiedler89}: we will see that the function $Z$ is non-increasing, and strictly decreasing if $d(u_0,v_0)>0$ for some $u_0,v_0$ in the support of a measure.

The technicality we need to resolve first is the possibility that $Z(\mu_0)=\infty$. It is not difficult to do it in the bounded case by using the ergodic decomposition of measures in $\tilde{\mathcal{B}}^2$. We first show in Lemma \ref{l:4_1} that if $\mu_0 \times \mu_0$ is $\hat{T}$-ergodic, then $\hat{Z}$ is finite. In general, in Lemma \ref{l:4_2} that we show that can always modify "weights" in the ergodic decomposition to make $\hat{Z}$ finite.
Proposition \ref{p:diszero1} will then follow from integrating the balance law of zeroes (\ref{r:bez!}), which by $S$-periodicity in the bounded case reduces to
 \begin{equation}
 z\circ\hat{T}-z=-d. \label{r:bezun}
 \end{equation}
 
\begin{lemma} \label{l:4_1}
Let $\nu_0$ be a Borel-probability measure on $\tilde{\mathcal{B}}^2$. If $\nu_0$ is $\hat{T}$-ergodic, then $\hat{Z}(\nu_0) < \infty$.
\end{lemma}

\begin{proof}
By ergodicity, any  $\hat{T}$-invariant set has $\nu_0$ measure $0$ or $1$. The balance law (\ref{r:bezun}) implies that $z \circ \hat{T} \leq z$  thus the sets $\mathcal{B}_n=\lbrace (u_0,v_0) \in \tilde{\mathcal{B}}^2, \: z(u_0,v_0) \leq n  \rbrace$, $n \geq 0$ an integer, are $\hat{T}$-invariant. By finiteness of $z$, $\tilde{\mathcal{B}}^2 = \cup_{n=0}^{\infty} \mathcal{B}_n$. As $\nu(\tilde{\mathcal{B}}^2)=1$, by continuity of probability we have that there exists $n_0 \geq 0$ such that $\mu(\mathcal{B}_{n_0})=1$, thus $\hat{Z}(\nu_0) \leq n_0$.
\end{proof}

\begin{lemma} \label{l:4_2}
	Assume $\nu_0$ is $\hat{T}$-invariant and $(u_0,v_0) \in \supp\nu_0$. Then there exists a $\hat{T}$-invariant $\tilde{\nu}_0$ such that $\hat{Z}(\tilde{\nu}_0) < \infty$ and $(u_0,v_0) \in \supp\tilde{\nu}_0$.
\end{lemma}

\begin{proof}
	First we find a sequence of $\hat{T}$-ergodic measures $\nu_k$ such that $(u_0,v_0)$ is in the closure of $\cup_{k=1}^{\infty}\supp\nu_k$. We do it e.g. by choosing any $\hat{T}$-ergodic measure $\nu_k$ such that the $\nu_k(\mathcal{B}_k)>0$, where $\mathcal{B}_k$ is the $1/k$-ball around $(u_0,v_0)$. Such a measure $\nu_k$ must exist by the ergodic decomposition theorem \cite{Walters}. Let 
	$$
	\tilde{\nu}_0 = \kappa \sum_{k=1}^{\infty} \frac{1}{\max \lbrace \hat{Z}(\nu_k), 2^k\rbrace } \nu_k,
	$$
	where $\kappa$ is uniquely chosen so that $\tilde{\nu}_0$ is a probability measure. Indeed, $1 \leq \kappa < \infty$, as by Lemma \ref{l:4_1}, the sum of the factors is 
	$$0 < \sum_{k=1}^{\infty} 1/\max \lbrace \hat{Z}(\nu_k), 2^k\rbrace \leq 1.$$ 
	Also by construction, $\hat{Z}(\tilde{\nu}_0) \leq \kappa < \infty$. We see that $(u_0,v_0) \in \supp\tilde{\nu}_0$ by choosing any $w_k \in \mathcal{B}_k \cap \supp \nu_k \subset \supp \nu_k \subset \supp\tilde{\nu}_0$. Then $w_k$ converges to $(u_0,v_0)$, so $(u_0,v_0)$ must be in $\supp\tilde{\nu}_0$ as the support of a measure is always closed.
\end{proof}


\begin{proof}[Proof of Proposition \ref{p:diszero1}] Assume $(u_0,v_0) \in \mathcal{E}^2$, i.e. that $u_0 \in \supp\mu_1$, $v_0 \in \supp\mu_2$ for some $T$-invariant $\mu_1$, $\mu_2$. Let $\nu_0 = \mu_1 \times \mu_2$, and let $\tilde{\nu}_0$ be a $\hat{T}$-invariant measure constructed in Lemma \ref{l:4_2}. We can iterate (\ref{r:bezun}) with respect to $\hat{T}$ and sum with (\ref{r:bezun}) to obtain
$$
z \circ \hat{T}^2 - z = - d - d \circ \hat{T}.
$$
Integrating it with respect to $\tilde{\nu}_0$ and using $\hat{T}$-invariance of $\tilde{\nu}_0$ and integrability of $\hat{Z}$ (and thus integrability of $\hat{Z}\circ \hat{T}$, $\hat{Z}\circ \hat{T}^2$), we see that
\begin{equation}
\int_{\tilde{\mathcal{B}}^2} d \: d\tilde{\nu}_0 + \int_{\tilde{\mathcal{B}}^2} d \circ \hat{T} \: d\tilde{\nu}_0 = 0.  \label{r:contra}
\end{equation}
	
Now, assume that $d(u_0,v_0) >0$. We now find an open neighbourhood $\mathcal{U}$ of $(u_0,v_0)$ such that (\ref{r:disseq}) holds. Then by $S$-invariance of all $u_0 \in \mathcal{B}$ in the bounded case,  we have that for all $(\tilde{u}_0,\tilde{v}_0) \in \mathcal{U}$, $d(\tilde{u}_0,\tilde{v}_0)+	d \circ \hat{T} (\tilde{u}_0,\tilde{v}_0) \geq 1$. As $(u_0,v_0)$ is in the support of $\tilde{\nu}_0$, we have that $\tilde{\nu}_0(\mathcal{U})\geq \varepsilon $ for some $\varepsilon > 0$. Now as always $d \geq 0$, we have
$$
\int_{\tilde{\mathcal{B}}^2} d \: d\tilde{\nu}_0 + \int_{\tilde{\mathcal{B}}^2} d \circ \hat{T}\: d\tilde{\nu}_0 \geq \varepsilon,  
$$		
which contradicts (\ref{r:contra}).
\end{proof}

\begin{proof}[Proof of Theorem \ref{t:main1}]
Let $(u_0(0),(u_0)_x(0))=(v_0(0),(v_0)_x(0))$ for some $u_0,v_0 \in\mathcal{E}$. By $T$-invariance of $\mathcal{E}$ and as $\mathcal{E}$ consists of entire solutions (see Lemma \ref{l:properties},(iii)), we have $(T^{-1}u_0,T^{-1}v_0)\in \mathcal{E}$, and by definition of $d$, we have $d(T^{-1}u_0,T^{-1}v_0)\geq 1$.  This is by Proposition \ref{p:diszero1} impossible. 
\end{proof}

\section{The proof of Theorem \ref{t:main1b} (the extended case)} \label{s:extended}

In this section, we consider only the E/P case (the E/A case is analogous). Intuitively, the proof is expected to be analogous to the proof of Theorem \ref{t:main1}: we assume that $d(u_0,v_0)> 0$ for some $u_0,v_0$ in the support of a $S,T$-invariant measure $\mu_0$. We then integrate the balance law (\ref{r:bez!}) with respect to $\mu_0 \times \mu_0$ and intuitively deduce that $d=0$, $\mu_0 \times \mu_0$-a.e.. This, however, would contradict $d(u_0,v_0)> 0$ and Lemma \ref{l:3_3}.

The main technical difficulty in this approach, however, is possible non-integrability of the zero and flux functions $z,f$. To address it in the extended case, we apply abstract ergodic-theoretical tools for two commuting transformations, in this case $\hat{S},\hat{T}$ on $\mathcal{B}^2$.

In the first sub-section we deal with possible non-integrability of $f$ in the case of integrable $z$, and prove a balance law of zeroes on average, i.e. that the flux in (\ref{r:bez!}) cancels out when (\ref{r:bez!}) is integrated with respect to a $\hat{S}$-invariant measure. In the second subsection we deal with possible non-integrability of $z$. Analogously, we consider properties of the average density of zeroes defined as
$$
\hat{\zeta}(u_0,v_0) = \liminf_{n \rightarrow \infty} \frac{1}{2n} \sum_{k=-n}^{n-1}z(S^ku_0,S^kv_0)
$$
By the Birkhoff ergodic theorem, for any $\hat{S}$-invariant measure $\nu_0$ on $\tilde{\mathcal{B}}^2$, for $\nu_0$-a.e. $(u_0,v_0)$, the $\liminf$ in the definition of $\hat{\zeta}$ can be replaced with $\lim$, though we can not exclude the possibility that the value of $\hat{\zeta}$ is $+\infty$. We then characterize the case of $\hat{\zeta}$ being $\nu_0$-a.e. finite. We use these tools to complete the proof of Theorem \ref{t:main1b} analogously as in the bounded case, by assuming in addition the non-degeneracy condition, i.e. that the average density of zeroes is finite.

\subsection{The balance law of zeroes on average} \label{ss:balaverage}

The main tool in this section is evolution of probability measures on $\tilde{ \mathcal{B}}$ with respect to (\ref{r:main}), and more generally evolution of measures on   $\tilde{ \mathcal{B}}^2$ with respect to two replicas of (\ref{r:main}) in the following sense. Assume $\nu_0$ is a Borel probability measure on  $\tilde{ \mathcal{B}}^2$. Then we denote by $\nu(t)$ the Borel probability measure on $\tilde{ \mathcal{B}}^2$, defined as $\nu_0$ pushed by product of two time-$t$ maps. If $\nu_0$ is $\hat{S}$-invariant, so is $\nu(t)$. We can define $\hat{Z}(\nu(t))$ as in Remark \ref{rr:61}.

We prove the following, using the aforementioned notation:
	
\begin{proposition} {\bf The balance law of zeroes on average.} \label{p:avbal} 
	Assume $\nu_0$ is a $\hat{S}$-invariant measure on $\tilde{\mathcal{B}}^2$, such that $\hat{Z}(\nu_0)< \infty$. Then
	\begin{equation}
	 \hat{Z}(\nu(0))=\hat{Z}(\nu(1)) + \int_{\tilde{\mathcal{B}}^2} d \: d\nu(0). \label{r:avbal}
	\end{equation}
\end{proposition}
The proposition will follow from a general ergodic theoretical argument below, which is required to deduce that the flux $f$ in (\ref{r:bez!}) cancels out when integrated with respect to a $\hat{S}$-invariant measure, even in the case when $f$ is not integrable.

\begin{lemma} \label{l:ergodic}
	Assume $(\Omega,\mathcal{F},\nu)$ is a probability space, that $\hat{\sigma}: \Omega \rightarrow \Omega$ is  measurable, and that $\nu$ is $\hat{\sigma}$-invariant. Assume that $\varphi, \zeta : \Omega \rightarrow \mathbb{R}$ are measurable, and that $\zeta$ is $\nu$-integrable. Furthermore, assume that $\nu$-a.e.,
	\begin{equation}
	\varphi \circ \hat{\sigma} -\varphi  \geq - \zeta . \label{r:assume}
	\end{equation}
	Then $(\varphi \circ \hat{\sigma} -\varphi )$ is $\nu$-integrable and $\int_{\Omega}(\varphi \circ \hat{\sigma} -\varphi )d\nu= 0$.
\end{lemma}

\begin{proof} Let $\mathcal{U}_m$ be the set of all $u \in \Omega$ such that $\varphi(\hat{\sigma}^n(u)) \leq m$ for infinitely many $n \in \mathbb{N}$. Then it is easy to see that $\mathcal{U}_m$ is $\hat{\sigma}$-invariant, and by the Poincar\'{e} recurrence theorem applied to sets $\lbrace u: \: \varphi(u) \leq m \rbrace$, that 
	\begin{equation}
	\nu \lbrace \bigcup_{m=1}^{\infty}  \mathcal{U}_m \rbrace = 1. \label{r:cup}
	\end{equation}
Consider functions 
\begin{align*} 
u \mapsto & h(u): = \varphi(\hat{\sigma}(u)) - \varphi(u) + \zeta(u), \\
u \mapsto & h_m(u) : = \mathbf{1}_{\mathcal{U}_m}(u)\lbrace  \varphi(\hat{\sigma}(u)) - \varphi(u) + \zeta(u) \rbrace \wedge m,
\end{align*}
where $\bf{1}_{\mathcal{U}_m}$ is the characteristic function and $\wedge$ the minimum. By the assumptions, $h \geq 0$, thus $h_m \geq 0$, and by construction and (\ref{r:cup}), $h_m$ is an increasing sequence of functions converging $\nu$-a.e. to $h$. 

We will first show that $\mathbb{E}[h_m] \leq \mathbb{E}[\zeta]$, where $\mathbb{E}[.]$ denotes the expectation, i.e. the Lebesgue integral with respect to $\nu$. Let $\mathcal{S}$ be the $\sigma$-algebra of $\hat{\sigma}$-invariant sets. It suffices to show that for all $m \geq 0$, 
\begin{equation}
\mathbb{E}[h_m|\mathcal{S}] \leq \mathbb{E}[\zeta|\mathcal{S}], \quad \nu-a.e., \label{r:ae}
\end{equation}
 where $\mathbb{E}[.|\mathcal{S}]$ denotes the conditional expectation \cite{Durrett:05}. As $0 \leq h_m \leq m$, $h_m$ is integrable, thus by the Birkhoff ergodic theorem, we have that $\nu$-a.e.,
\begin{equation}
\lim_{n \rightarrow \infty} \frac{1}{n}\sum_{k=0}^{n-1}h_m \circ \hat{\sigma}^k = \mathbb{E}[h_m|\mathcal{S}]. \label{r:birkhoff}
\end{equation} 
Without loss of generality $\nu(\mathcal{U}_m) > 0$ (otherwise $h_m=0$ $\nu$-a.e.).	Choose $u \in \mathcal{U}_m$, and one of infinitely $n_j$ such that $\varphi(\hat{\sigma}^{n_j}(u))\leq m$. Then it is easy to see that
\begin{align}
 \frac{1}{n_j}\sum_{k=0}^{n_j-1}h_m(\hat{\sigma}^k(u)) & \leq \frac{1}{n_j}\sum_{k=0}^{n_j-1}h(\hat{\sigma}(u)) = \frac{1}{n_j}(\varphi(\hat{\sigma}^{n_j}(u))-\varphi(u)) + \frac{1}{n_j}\sum_{k=0}^{n_j-1}\zeta (\hat{\sigma}^k(u)) \notag \\
 & \leq \frac{1}{n_j}(m-\varphi(u)) + \frac{1}{n_j}\sum_{k=0}^{n_j-1}\zeta (\hat{\sigma}^k(u)). \label{r:birkhoff2}
\end{align}
Now by the Birkhoff ergodic theorem applied to $\zeta$, we see that the right-hand side of (\ref{r:birkhoff2}) converges to $\mathbb{E}[\zeta|\mathcal{S}]$ as $n_j \rightarrow \infty$. Combined with (\ref{r:birkhoff}), we deduce that for $\nu$-a.e. $u \in \mathcal{U}_m$, we have that $\mathbb{E}[h_m|\mathcal{S}] \leq \mathbb{E}[\zeta|\mathcal{S}]$. As for $u \in \mathcal{U}^c_m$, $h_m(u) = 0$ and $\mathcal{U}^c_m$ is $\hat{\sigma}$-invariant, we conclude that (\ref{r:ae}) holds also for $\nu$-a.e. $u \in \mathcal{U}^c_m$.

Now, by the definition of the conditional expectation, (\ref{r:ae}) implies that for all $m \in \mathbb{N}$, $\mathbb{E}[h_m] \leq \mathbb{E}[\zeta]$, thus by the Lebesgue monotone convergence theorem, $h$ is integrable and $\mathbb{E}[h] \leq \mathbb{E}[\zeta]$. As we can now apply the Birkhoff ergodic theorem also to $h$, we repeat the argument as in (\ref{r:birkhoff}) and (\ref{r:birkhoff2}) applied to $h$ instead of $h_m$ to conclude that $\mathbb{E}[h] = \mathbb{E}[\zeta]$. As now $h-\zeta$ is integrable and $\mathbb{E}[h-\zeta]=0$, the proof is complete.
\end{proof}

\begin{proof}[Proof of Proposition \ref{p:avbal}] 
	We insert in Lemma \ref{l:ergodic} the following: $\Omega=\tilde{\mathcal{B}}^2$ with the Borel $\sigma$-algebra, $\hat{\sigma} = \hat{S}$, $\varphi=f$, $\zeta = z$ and $\nu=\nu_0$.  By (\ref{r:bez!}), we have 
	$$ f \circ \hat{S}- f = z \circ \hat{T} - z + d \geq -z,$$ 
	thus the assumptions of Lemma \ref{l:ergodic} hold. We now have that $( f \circ \hat{S} - f)$ is $\nu_0$-integrable and
	$$
	\int_{\tilde{\mathcal{B}}^2} \left( f\circ \hat{S} - f \right) d\nu_0 = 0.
	$$
	Inserting it into (\ref{r:bez!}) integrated with respect to $\nu_0$, we obtain (\ref{r:avbal}).  
\end{proof}

\subsection{Density of zeroes and non-degeneracy of invariant measures} \label{ss:characterize} We now prove several properties of the density of zeroes. We first establish that the density of zeroes is a.e. non-decreasing, and then define and characterize non-degeneracy of invariant measures.
\begin{lemma} \label{l:nondecreasing}
	Assume $\nu_0$ is a $\hat{S}$-invariant measure on $\tilde{\mathcal{B}}^2$. Then for $\nu_0$-a.e. $(u_0,v_0)$, 
	\begin{equation}
	\hat{\zeta}(Tu_0,Tv_0)\leq \hat{\zeta}(u_0,v_0). \label{r:nondecreasing}
	\end{equation} 
\end{lemma}

\begin{proof}
	It suffices to prove that (\ref{r:nondecreasing}) holds a.e. with respect to every $\hat{S}$-ergodic measure $\nu_0$, as the claim then follows by the ergodic decomposition theorem. This follows from the Birkhoff ergodic theorem and (\ref{r:avbal}) if $\hat{Z}(\nu_0) < \infty$, and trivially if $\hat{Z}(\nu_0) = \infty$, as then $\hat{\zeta}(u_0,v_0) = \infty$ $\nu_0$-a.e.
\end{proof}

\begin{defn} \label{d:nondegenerate} We say that a $\hat{S}$-invariant measure $\nu_0$ on $\tilde{\mathcal{B}}^2$ is non-degenerate, if for $\nu$-a.e. $(u_0,v_0)$, $\hat{\zeta}(u_0,v_0)<\infty$.
We say that a pair $(\mu_1,\mu_2)$ of $S$-invariant measures on $\tilde{\mathcal{B}}$ is non-degenerate, if $\mu_1 \times \mu_2$ is non-degenerate. A $S$-invariant measure $ \mu_0$ on  $\tilde{\mathcal{B}}$ is non-degenerate, if the pair $(\mu_0,\mu_0)$ is non-degenerate. A family of $S$-invariant measures $\mathcal{N}$ on $\tilde{\mathcal{B}}$ is non-degenerate, if every $\mu_0 \in \mathcal{N}$ is non-degenerate. The ergodic attractor $\mathcal{E}$ is non-degenerate, if $\mathcal{M}(\mathcal{B})$ is non-degenerate.
\end{defn}

We note that we do not know of any examples of degenerate measures on $\tilde{\mathcal{B}}^2$. We discuss it further in Section \ref{s:open}. 

In the following lemma, we use the ergodic decomposition of a measure with respect to two commuting transformations. We say that a measure is ergodic with respect to two commuting transformations $\hat{S},\hat{T}$, if any $\hat{S},\hat{T}$-invariant measurable set has measure $0$ or $1$. We can decompose a $\hat{S},\hat{T}$-invariant measure on $\tilde{\mathcal{B}}^2$ into $\hat{S},\hat{T}$-ergodic measures, with the standard decomposition formula \cite{Walters}, Section 6.2, as the Choquet theorem applies. We will require the following generalization of Lemma \ref{l:4_1} to the extended case.

\begin{lemma} \label{l:nondeg} Let $\nu_0$ be a $\hat{S}$-invariant measure on $\tilde{\mathcal{B}}^2$.
	
	(i) $\nu_0$ is non-degenerate if and only if for a.e. measure $\nu_1$ in its ergodic decomposition into $\hat{S}$-ergodic measures, $\hat{Z}(\nu_1)< \infty$.

	(ii) Assume $\nu_0$ is $\hat{S},\hat{T}$-invariant. Then $\nu_0$ is non-degenerate if and only if for a.e. measure $\nu_1$ in its ergodic decomposition into $\hat{S},\hat{T}$-ergodic measures, $\hat{Z}(\nu_1)< \infty$.  
\end{lemma}

\begin{proof}
	Assume $\nu_0$ is non-degenerate, and take any measure $\nu_1$ from its $\hat{S}$-ergodic decomposition such that $\hat{\zeta}< \infty$ $\nu_1$-a.e. (this holds for a.e. measure in the ergodic decomposition.) As for each $n$, the set $\lbrace (u_0,v_0), \: \hat{\zeta}(u_0,v_0) \leq n \rbrace$ is $\hat{S}$-invariant, it has $\nu_0$-measure $0$ or $1$, thus we can find $n_1$ large enough such that $\nu_1(\hat{\zeta} \leq n_0 )=1$, so $\hat{Z}(\nu_1)\leq n_0$. The other implication in (i) follows from the ergodic decomposition theorem.
	
	To show (ii), it suffices to note that by Lemma \ref{l:nondecreasing}, for every $\hat{S},\hat{T}$-ergodic measure $\nu_0$, the sets $\lbrace (u,v), \: \hat{\zeta}(u_0,v_0) \leq n \rbrace$ are $\nu_0$-a.e. $\hat{S},\hat{T}$-invariant. The rest of the proof is analogous to the case (i).
\end{proof}

\subsection{Proof of Theorem \ref{t:main1b}} We prove the following slightly generalized version of Theorem \ref{t:main1b}:

\begin{proposition} \label{p:final} Assume $\mathcal{M}_0(\mathcal{B})$ is a non-degenerate family of $S,T$-invariant measures, closed for finite or countable convex combinations, and let $\mathcal{E}_0 = \cup_{\mu \in \mathcal{M}_0(\mathcal{B})}\supp \mu$. Then for any $(u_0,v_0) \in \mathcal{E}_0 $, we have that $d(u_0,v_0)=0$. 
\end{proposition}

\begin{proof} Let $u_0 \in \supp \mu_1$, $v_0 \in \supp \mu_2$, and let $\nu_0 = \frac{1}{4}(\mu_1+\mu_2)^2$. Then $\nu_0$ is a $\hat{S},\hat{T}$-invariant measure on $\tilde{\mathcal{B}}^2$, by assumptions non-degenerate. Analogously as in Lemma \ref{l:4_2}, by applying Lemma \ref{l:nondeg}, we can construct a $\hat{S},\hat{T}$-invariant $\tilde{\nu}_0$ such that $\hat{Z}(\tilde{\nu}_0) < \infty$, and such that $(u_0,v_0) \in \supp \tilde{\nu}_0$. As $\tilde{\nu}$ is $\hat{S},\hat{T}$-invariant, (\ref{r:avbal}) implies that $d=0$, $\tilde{\nu}$-a.e. The rest of the argument is analogous to the proof of Proposition \ref{p:diszero1}.
\end{proof}

\begin{proof}[Proof of Theorem \ref{t:main1b}] It is analogous to the proof of Theorem \ref{t:main1}, 
by applying Proposition \ref{p:final} instead of Proposition \ref{p:diszero1} to $\mathcal{M}_0(\mathcal{B})$ instead of $\mathcal{M}(\mathcal{B})$.
\end{proof}

\part{Uniqueness of invariant measures}

\section{Non-transversal intersections of an equilibrium in the extended case} \label{s:notrans}

Prior to discussing uniqueness of an invariant measure, we demonstrate here a universal property of non-transversality of intersections of $\omega$-limit sets almost-everywhere and a $S,T$-equilibrium (i.e. a spatially and temporally periodic solution). We consider only the E/P case in this section, assume (A1-4), and let $\mathcal{B}$, $\tilde{\mathcal{B}}$ and $\delta_0$ be as in Section \ref{s:prelim}. We fix throughout the section a $v_0 \in \mathcal{B}$ such that $v_0 = S(v_0)=T(v_0)$ (thus $v_0 \in \tilde{\mathcal{B}}$).
Recall that the pair $(\mu_0,\delta_{v_0})$ is non-degenerate, if for $\mu_0$-a.e. $u_0$, (\ref{d:nondegenerate}) holds, i.e. the density of zeroes of $u_0-v_0$ is finite. We do not know any example of an $S$-invariant $\mu_0$ supported on $\tilde{\mathcal{B}}$ and a $S,T$-invariant $v_0$ such that $(\mu_0,\delta_{v_0})$ is degenerate.

Throughout the section, $u(t),v(t),z(t)$ denote solutions of (\ref{r:main}) with initial conditions $u_0,v_0,z_0$ at $t=0$.

\begin{proposition} \label{p:transversal}
	Assume that $\mu_0$ is an $S$-invariant measure on $\tilde{\mathcal{B}}$ such that $(\mu_0,\delta_{v_0})$ is non-degenerate. Then for $\mu_0$-a.e. $u_0$, $\omega(u_0)$ consists of $z_0$ such that $z(t)-v(t)$ can not have a multiple zero for any $(x,t) \in \mathbb{R}^2$.
\end{proposition}

\begin{proof} First note that it suffices to prove the claim for $S$-ergodic $\mu_0$, as by the assumption and Lemma \ref{l:nondeg}, every measure $\nu_0$ in the $\hat{S}$-ergodic decomposition of $\mu_0 \times \delta_{v_0}$ is non-degenerate, and a.e. measures in the $\hat{S}$-ergodic decomposition of $\mu_0 \times \delta_{v_0}$ are of the form $\mu_1 \times \delta_{v_0}$, $\mu_1$ $S$-ergodic. Thus assume $\mu_0$ is $S$-ergodic, so by the non-degeneracy assumption and Lemma \ref{l:nondeg} we have $\hat{Z}(\mu_0 \times \delta_{v_0}) < \infty$.

We will first show that there exists an open set $\mathcal{U} \subset \tilde{\mathcal{B}}$ satisfying
\begin{subequations} \label{r:U2}
\begin{gather}
 \left\lbrace u_0 \in \tilde{\mathcal{B}}, \: d(u_0,v_0)\geq 1 \right\rbrace  \subset \mathcal{U}, \label{r:U2a} \\  \mathcal{U} \subset \left\lbrace u_0\in \tilde{\mathcal{B}}, \: d(u_0,v_0) + d(S^{-1}u_0,v_0)+d(Tu_0,v_0)+d(S^{-1}Tu_0,v_0)\geq 1 \right\rbrace. 
   \label{r:U2b}
\end{gather}
\end{subequations} 
Then we show that 
\begin{equation}
\sum_{k=1}^{\infty}\mu_0 \left( T^{-k} \left( \mathcal{U} \right) \right) < \infty, \label{r:borelcantelli}
\end{equation}
and finally we complete the proof by an application of the first Borel-Cantelli lemma \cite{Durrett:05}.

To prove the first claim, for any $z_0$ such that $d(z_0,v_0) \geq 1$ we can by an application of Lemma \ref{l:3_3} find an open neighbourhood $\tilde{\mathcal{U}}(z_0) \subset \tilde{\mathcal{B}}$ such that for each $\tilde{u}_0 \in \tilde{\mathcal{U}}(z_0)$, and for $\tilde{v}_0=v_0$, (\ref{r:disseq}) holds, thus as $v_0=Tv_0=S^{-1}v_0$, 
\begin{equation}
d(\tilde{u}_0,v_0)+d(S^{-1}\tilde{u}_0,v_0)+d(T\tilde{u}_0,v_0)+d(S^{-1}T\tilde{u}_0,v_0)\geq 1. \label{r:geq1}
\end{equation}
The set $\mathcal{U}=\cup_{z_0 \in \tilde{\mathcal{B}},d(z_0,v_0)\geq 1} \tilde{\mathcal{U}}(z_0)$ now satisfies (\ref{r:U2}). 

Let $\nu(t)$ be the evolution of the measure $\mu_0 \times \delta_{v_0}$ as defined at the beginning of the subsection \ref{ss:balaverage}. Then for integers $k \geq 0$, $\nu(k)=\mu(k) \times  \delta_{v_0}$ by $S,T$-invariance of $v_0$, where $\mu(t)$ is the evolution of $\mu_0$. Now applying the balance law on average \ref{r:avbal} to $\nu(k)$, we see that
\begin{equation}
 \sum_{k=0}^{\infty} \int_{\tilde{\mathcal{B}}} d(T^k(u_0),v_0) d\mu_0(u_0)
 = \sum_{k=0}^{\infty} \int_{\tilde{\mathcal{B}}} d \: d\nu(k) \leq \hat{Z}(\nu(0)) = 
 \hat{Z}(\mu_0 \times \delta_{v_0}) < \infty. \label{r:borel1}
\end{equation}
By (\ref{r:U2b}) we obtain
\begin{align} \mu_0(T^{-k}(\mathcal{U})) & = \mu_0( u_0 \in \mathcal{B}, \, T^k(u_0)\in \mathcal{U}) \notag\\
&  \leq \int_{\tilde{\mathcal{B}}} \left( d\left( T^ku_0,v_0 \right) + d\left( S^{-1}T^ku_0,v_0 \right)+d\left( T^{k+1}u_0,v_0 \right)+d\left(S^{-1}T^{k+1}u_0,v_0 \right) \right) d\mu_0(u_0) \notag \\
 & =2 \int_{\tilde{\mathcal{B}}} \left( d\left(T^ku_0,v_0 \right) + d\left(T^{k+1}u_0,v_0\right) \right)  d\mu_0(u_0), \label{r:borel2}
\end{align}
where in the last row we applied the $S$-invariance of $\mu_0$. Inserting (\ref{r:borel2}) into (\ref{r:borel1}) we obtain (\ref{r:borelcantelli}). By the first Borel-Cantelli lemma \cite[(6.1)]{Durrett:05}, the set of $u_0 \in \tilde{\mathcal{B}}$ such that $T^k(u_0)\in \mathcal{U}$ for infinitely many $k \in \mathbb{N}$ has $\mu_0$-measure $0$, thus by openness of $\mathcal{U}$, $\mu_0 \left( \left\lbrace u_0\in \tilde{\mathcal{B}}, \: \omega(u_0)\cap \mathcal{U}=\emptyset \right\rbrace \right) = 1$. Because of (\ref{r:U2a}), we obtain
$$
\mu_0 \left( \left\lbrace u_0\in \tilde{\mathcal{B}}, \: \forall z_0 \in \omega(u_0), \: \forall k \in \mathbb{Z}, \, d(z(k),v(k))=0 \right\rbrace \right) = 1
$$
(where $v(k)=v_0$ by $T$-invariance of $v_0$). Finally, as $\mu_0$ is $S$-invariant, we get
$$
\mu_0 \left( \left\lbrace u_0\in \tilde{\mathcal{B}}, \: \forall z_0 \in \omega(u_0), \: \forall k,m \in \mathbb{Z}, \, d(S^mz(k),S^mv(k))=0 \right\rbrace \right) = 1
$$
(where $S^mv(k)=v_0$ by $S,T$-invariance of $v_0$), which we needed to prove by the definition of $d$.
\end{proof}

\begin{remark} \label{e:nondeg} One can construct many non-trivial measures $\mu_0$ such that $(\mu_0,\delta_{v_0})$ is non-degenerate for any $S,T$-invariant $v_0$, without any a-priori information on $v_0$. For example, this can be done by embedding measures by combining profiles of a finite family of $S$-invariant (i.e. spatially periodic) functions in $\tilde{ \mathcal{B}}$, as $z(u_0,v_0) < \infty$ for any $u_0,v_0 \in \tilde{\mathcal{B}}$. Such $\mu_0$ then also satisfies (N1) without any a-priori knowledge of $\mathcal{V}$.
	
\end{remark}

\section{Existence of a 1d family of equilibria} \label{s:burgers}

We now prove the second step in the proof of the results listed in Subsection \ref{ss:unique}: existence of a 1d family of equilibria as specified in Theorem \ref{t:main2}, (i). More specifically, let $1 > \alpha > 1 - \varepsilon / 2$, where $\varepsilon$ is as in (B1). We will construct a 1d family satisfying the following:

\begin{itemize}
	\item[(C1)] (i)  There exists a set ${\mathcal{V}}=\lbrace v^y_0, \: y \in \mathbb{R}$, $v^y_0 \in H^{2 \alpha}(\mathbb{S}^1)$ satisfying that $y \mapsto v^y_0$ is continuous as a map $\mathbb{R} \rightarrow H^{2 \alpha}(\mathbb{S}^1)$, strictly increasing, and such that for all $y \in \mathbb{R}$, $v^y_0$ is $T$-invariant and $\int_0^1 v^y_0(x)dx=y$.
	\vspace{1ex}
	
	\noindent (ii) The functions 
	\begin{align*}
	y \mapsto \underline{v}(y) & := \min \left\lbrace v^y(x,t), \: (x,t) \in \mathbb{S}^1 \times [0,1] \right\rbrace, \\
	y \mapsto \overline{v}(y) & := \max \left\lbrace v^y(x,t), \: (x,t) \in \mathbb{S}^1 \times [0,1] \right\rbrace
	\end{align*} 
	are onto $\mathbb{R}$.
	
	\vspace{1ex}
	\noindent (iii) Such a set $\mathcal{V}$ is unique.
\end{itemize}

We prove the following, by an application of the Schauder fixed point theorem:

\begin{proposition} \label{l:C1}
	If (A1-3) and (B1-3) hold, then (C1) holds.
\end{proposition}

We first prove uniqueness in a separate lemma.

\begin{lemma} \label{l:unique} If (A1-3) holds, then a set $\mathcal{V}$ satisfying (C1) (i),(ii) is unique.
\end{lemma}	

\begin{proof}
	Assume the contrary, and find two such families $v^y_0$,$w^y_0$. Then it is easy to see that there must exist two $y_1,y_2$ such that $d(v^{y_1}_0,w^{y_2}_0) \geq 1$, for example by choosing any $y_1$ such that $v^{y_1}_0 \neq w^{y_1}_0$ and setting $y_2$ to be the minimal $y$ such that $w^y_0 \geq v^{y_1}_0$. Clearly $\mathcal{B}$ consisting of orbits $v^{y_1}(t),w^{y_2}(t)$ satisfies (A4), and the measures $\delta_1$, $\delta_2$ concentrated on $v^{y_1}_0,w^{y_2}_0$ are $S,T$-invariant. This and $d(v^{y_1}_0,w^{y_2}_0) \geq 1$ contradicts Theorem \ref{t:main1}.
\end{proof}

\begin{proof}[Proof of Proposition \ref{l:C1}] Throughout the proof, we consider the dynamics of (\ref{r:main}) in the bounded case $\mathcal{X}=L^2(\mathbb{S}^1)$ only, and assume (A1-3), (B1-3). Fix $n \in \mathbb{N}$ and a function $c : [\alpha,1) \rightarrow (n,\infty)$ to be chosen later, and consider the family $\mathcal{V}_{n,c}$ of continuous functions $w_0 : [-n,n] \rightarrow \mathcal{X}^{\alpha}$, $ y \mapsto w^y_0$, satisfying the following properties for all $y,z \in  [-n,n]$:
	\begin{subequations}
		\begin{align}
		& \int_0^1 w^y_0(x)dx   =y, \label{r:schA} \\
		& \left|\left|w^y_0-y \right|\right|_{L^{\infty}(\mathbb{S}^1)} \leq l(y), \label{r:schB} \\
		&	 y \leq z  \Rightarrow w^y_0 \leq w^z_0, \label{r:schC} \\
		& \left|\left|w^y_0 \right|\right|_{\mathcal{X}^{\gamma}} \leq c(\gamma)\quad \text{for all } \gamma \in [\alpha,1), \label{r:schD}
		\end{align}
	\end{subequations}
	where $l(y)$ is as in (B2). Clearly $\mathcal{V}_{n,c}$ is convex. We need to also show that it is compact in $C([-n,n],\mathcal{X}^{\alpha})$ and non-empty.
	
	First note that by (\ref{r:schA}) and (\ref{r:schC}), for $y < z$ we have that $||w^z_0-w^y_0||_{L^1(\mathbb{S}^1)} = z - y$, thus
	\begin{equation}
	||w^z_0-w^y_0||_{L^2(\mathbb{S}^1)} \leq c_{\infty}^{1/2} (z - y)^{1/2}. \label{r:l2}
	\end{equation}
	Fix a $\gamma$, $\alpha < \gamma < 1$. By the interpolation formula \cite{Henry}, p27, we have $||u||_{\mathcal{X}^{\alpha}} \leq c_1 ||u||_{\mathcal{X}^{\gamma}}^{\alpha/\gamma} ||u||_{\mathcal{X}}^{1-\alpha/\gamma}$ for some fixed constant $c_1 > 0$, thus by (\ref{r:schD}) and (\ref{r:l2}),
	\begin{align*}
	||w^z_0-w^y_0||_{\mathcal{X}^{\alpha}} \leq 2^{\alpha/\gamma} c_1 c_{\infty}^{1/2-\alpha/(2\gamma)} c(\gamma)^{\alpha/\gamma} |z-y|^{1/2-\alpha/(2\gamma)}.
	\end{align*}
	We see that $\mathcal{V}_{n,c}$ is equicontinuous, thus by the Arzel\`{a}-Ascoli theorem, its closure is compact.
	
	To show that $\mathcal{V}_{n,c}$ is compact, it remains to show that it is closed in $C([-n,n],\mathcal{X}^{\alpha})$. The only remaining non-trivial claim is that it is closed with respect to (\ref{r:schD}) for $\alpha < \gamma < 1$. It suffices to show that if for some $y \in [-n,n]$, a sequence $w_n^y$ satisfying (\ref{r:schD}) converges in $\mathcal{X}^{\alpha}$ to $z^y_0 \in \mathcal{X}^{\alpha}$, that then $z^y_0 \in \mathcal{X}^{\gamma}$ and $||z^y_0||_{\mathcal{X}^{\gamma}} \leq c(\gamma)$. By taking some $\gamma' > \gamma$, by compact embedding of $\mathcal{X}^{\gamma'}$ in $\mathcal{X}^{\gamma}$ we deduce that the family $w_n^y$, $n \in \mathbb{N}$ is relatively compact in $\mathcal{X}^{\gamma}$. As its every convergent subsequence in $\mathcal{X}^{\gamma}$ converges also in $\mathcal{X}^{\alpha}$, it must converge to $z^y_0$, thus $z^y_0 \in \mathcal{X}^{\gamma}$, $w_n^y$ converges to $z^y_0$ in $\mathcal{X}^{\gamma}$ and (\ref{r:schD}) holds in the limit.
	
	We now show that for each $n \in \mathbb{N}$, there exists a function $c$ as in (\ref{r:schD}) such that $\mathcal{V}_{n,c}$ is non-empty, and such that the function $\tau : \mathcal{V}_{n,c} \rightarrow \mathcal{V}_{n,c}$ given with $(\tau(w_0))^y = T(w^y_0)$ is well defined, i.e. that $\tau(w_0) \in \mathcal{V}_{n,c}$. The properties (\ref{r:schA}) and (\ref{r:schB}) are preserved by (B2), (B3) respectively; and (\ref{r:schC}) by the order-preserving property of (\ref{r:main}). To show $\tau$-invariance of (\ref{r:schD}), consider $c_\infty : = \max_{y \in [-n,n]}(|y|+l(y))$ (which exists by the upper semi-continuity of $l$). Then by (B1), (B2) and \cite[Proposition 7.2.2]{Lunardi:95}, the solution $w^y(t)$ of (\ref{r:main}), $w^y(0)=w^y_0$ exists for all $t \geq 0$ as long as $w^y_0$ satisfies (\ref{r:schB}), and for all $t \geq 0$, we have that $||w^y(t)||_{L^{\infty}(\mathbb{R})} \leq c_{\infty}$. Furthermore, by \cite[Lemma 7.0.3 and Proposition 7.2.2]{Lunardi:95}, we can find $c(\alpha)> n$ large enough, such that if $||w^y_0||_{\mathcal{X}^\alpha} \leq c(\alpha)$, then $||T(w^y_0)||_{\mathcal{X}^\alpha} \leq c(\alpha)$. Finally, we obtain the required $c(\gamma) > n$ for each $\alpha < \gamma < 1$ by integrating the variation of constants formula over $t \in [0,1]$ while applying (B1) and a-priori bounds on the solution in $\mathcal{X}^\alpha$ for $t \in [0,1]$ obtained in \cite[Proposition 7.2.2]{Lunardi:95},  Clearly now for $w^y_0 \equiv y$ we have that $w_0 \in \mathcal{V}_{n,c}$, thus $\mathcal{V}_{n,c}$ is non-empty. Finally, by the continuous dependence on initial conditions, $\tau(w)$ is continuous $\tau : \mathcal{V}_{n,c} \rightarrow \mathcal{V}_{n,c}$ is continuous. 
	
	Now we can apply the Schauder fixed point theorem to find a fixed point of $\tau$, which was required. We can extend $w^y$ to the entire $y \in \mathbb{R}$ by choosing an increasing sequence of $n_k \in \mathbb{N}$ such that $n_k > \max_{y \in [-n_{k-1},n_{k-1}]}(|y|+l(y))$, and proving that $\lbrace w^y, \: y \in [-n_j,n_j] \rbrace$ is then independent of $n_k$, $k > j$, analogously as in the proof of Lemma \ref{l:unique}. This completes (C1),(i). We obtain (C2),(ii) from (\ref{r:B2}) and the construction.	 
\end{proof}

\section{A 1d family of equilibria as an ergodic attractor and asymptotics} \label{s:unique}

We now complete the proofs of the main results of Subsection \ref{ss:unique}, which follow from Proposition \ref{l:C1} and Proposition \ref{p:C1} below. We actually show that the condition (C1) from the previous section suffices instead of (B2-3). Let $\alpha > 1 - \varepsilon / 2$, where $\varepsilon$ is as in (B1).

\begin{proposition} \label{p:C1}
If we assume (A1-3), (B1) and (C1), then the claims in Theorem \ref{c:bounded} and Corollaries \ref{c:bounded} and \ref{c:asymptotics} hold.
\end{proposition}

We prove it in a series of Lemmas, with standing assumptions (A1-3), (B1) and (C1). We first establish that $\mathcal{X}^{\alpha}$ can be decomposed into an increasing union of sets on which (A4) holds, and then deduce the required claims.

\begin{lemma} \label{l:unbound} 
 Assume $u_0 \in \mathcal{X}^{\alpha}$ in either bounded or extended case, such that $||u_0||_{\mathcal{X}^{\alpha}} \leq c_0$. Then there exists a constant $c_1 > 0$ (depending on $c_0$, non-linearity $g$ and family $\mathcal{V}$), such that for all $t_0 \in \mathbb{R}$, the solution $u(t)$ of (\ref{r:main}), $u(t_0)=u_0$, exists for all $t \geq t_0$ and $||u(t)||_{\mathcal{X}^{\alpha}} \leq c_1$.
\end{lemma}

\begin{proof}
By (C1),(ii), we can find $y_1 < y_2$ such that $\overline{v}(y_1) \leq u \leq \underline{v}(y_2)$. By the maximum principle, if the solution of (\ref{r:main}) exists on the interval $[t_0,t_1)$, then for each $t \in [t_0,t_1)$, we have that $\underline{v}(y_1) \leq u(t) \leq \overline{v}(y_2)$, thus $u(t)$ is uniformly bounded in the $L^{\infty}(\mathbb{S}^1)$, resp. $L^{\infty}(\mathbb{R})$ norm in the B, resp. E case. This and (C1) imply the claim by the standard argument, e.g. \cite[Proposition 7.2.2]{Lunardi:95} (alternatively, see \cite{Polacik:02}, Section 2). This in the view of the comments in the Appendix \ref{s:fractional} also holds in the extended case.
\end{proof}

Let $\tilde{\mathcal{B}}_k$ be the set of all $u_0 \in \tilde{ \mathcal{B}}$ such that for all $t \geq 0$, $||u(t)||_{\mathcal{X}^{\alpha}} \leq k$, where $u(t)$ is the solution of (\ref{r:main}), $u(0)=u_0$. Then by Lemma \ref{l:unbound}, $\tilde{ \mathcal{B}}=\cup_{k=1}^{\infty} \tilde{\mathcal{B}}_k$, and by the discussion in Section \ref{s:prelim}, $\tilde{\mathcal{B}}_k$ is compact and invariant. In this section we write $\mathcal{E}=\cup_{k=1}^{\infty} \mathcal{E}(\tilde{\mathcal{B}}_k)$.

\begin{lemma} \label{l:ii} In the bounded case, and in the extended case if $\mathcal{E}$ is non-degenerate, we have that $\mathcal{E}=\mathcal{V}$.
\end{lemma}

\begin{proof} Consider first the bounded case. Fix $k \in \mathbb{N}$, and consider
	$$
		\tilde{\mathcal{B}} := \tilde{\mathcal{B}}_k \cup \lbrace v^y_0, \: y \in [y^-,y^+] \rbrace,
	$$
	$\tilde{\mathcal{B}} \subset H^{2\alpha}(\mathbb{S}^1)$, where $y^- <y^+$ were chosen so that for all $u_0 \in \tilde{\mathcal{B}}_k$, $v^{y^-}_0 < u_0 < v^{y^+}_0$. This is possible, as by definition, $\tilde{\mathcal{B}}_k$ is uniformly bounded in $L^{\infty}(\mathbb{R})$, and because of (C1),(ii). 
	
	Clearly $\left\lbrace v^y_0, y\in [y^-,y^+] \right\rbrace \subseteq \mathcal{E}(\tilde{\mathcal{B}}_k)$, as the Dirac measure $\delta_{v^y_0}$ is $T$-invariant. Assume $\mu_0$ is any $T$-invariant measure on $\tilde{\mathcal{B}}$, and let $u_0 \in \supp \mu$. Let $y_1 < y_2$ be chosen so that $y_1 = \max \lbrace y \in \mathbb{R}, v^y_0 \leq u_0 \rbrace$, and $y_2 = \min \lbrace y \in \mathbb{R}, u_0 \leq v^y_0 \rbrace$ (such minimum and maximum exist by the compactness of the domain $\mathbb{S}^1$). If $y_1 \neq y_2$, we easily see that both $u_0-v^{y_1}_0$ and $u_0-v^{y_2}_0$ have a multiple zero, which is impossible by Proposition \ref{p:diszero1}.The only possibility is $u_0=v^{y_1}_0=v^{y_2}_0$, thus $u_0 \in \mathcal{V}$.
	
	Consider now the extended case with the non-degeneracy assumption, with $\tilde{\mathcal{B}}$ as above, thus now $\tilde{\mathcal{B}} \subset H^{2\alpha}_{\text{ul}}(\mathbb{R})$. Again we see that $\lbrace v^y_0, y\in [y^-,y^+] \rbrace \subseteq \mathcal{E}(\tilde{\mathcal{B}}_k)$, as the Dirac measure $\delta_{v^y_0}$ is $S,T$-invariant. Let $\mu_0$ be any $S,T$-invariant measure on $\tilde{\mathcal{B}}$, and let $u_0 \in \supp \mu_0$. Now suppose that $u_0$ intersects some $v^{y_0}_0$ twice at $x_1 < x_2$. Find $y_1 < y_2$ so that 
\begin{align*}
y_1 & = \max \left\lbrace y \in \mathbb{R}, \, v^y_0(x) \leq u_0(x), \, x\in [x_1,x_2] \right\rbrace, \\ y_2 & = \min \left\lbrace y \in \mathbb{R}, \, u_0(x) \leq v^y_0(x), \, x\in [x_1,x_2] \right\rbrace
\end{align*}
(such minimum and maximum exist by compactness of $[x_1,x_2]$). Thus by Proposition \ref{p:final}, we deduce analogously as in the bounded case that the only possibility is $u_0|_{[x_1,x_2]}=v^{y_1}_0|_{[x_1,x_2]}=v^{y_2}_0|_{[x_1,x_2]}$, thus by the local structure of zeroes, $u_0=v^{y_1}_0=u^{y_2}_0$, i.e. $u_0\in \mathcal{E}$. We conclude that $u_0$ can intersect every $v_0 \in \mathcal{V}$ at most once, transversally, so it is easy to see that the only alternative to $u_0 \in \mathcal{V}$ is $u_0 \in \mathcal{H}$, $\mathcal{H}$ the set of spatially heteroclinic solutions defined in the Introduction. By definition, no $h_0 \in \mathcal{H}$ is $S$-recurrent, thus by the Poincar\'{e} recurrence theorem, $\mu_0(\mathcal{H})=0$, thus $\mu_0(\mathcal{V})=1$. As $\mathcal{V}$ is a closed set, $\mu_0$ must be supported on $\mathcal{V}$, which eliminates the possibility $u_0 \in \mathcal{H}$ and concludes the proof also in the extended case.		
\end{proof}

%

\begin{lemma} \label{l:bbounded}
For each $u_0 \in \mathcal{X}^{\alpha}$ in the bounded case, there exists $y_0 \in \mathbb{R}$ such that $\omega(u_0)=\lbrace v^{y_0}_0 \rbrace$. 
\end{lemma}

\begin{proof}
	As $\bar{\omega}(u_0)$ is by Lemma \ref{l:onaverage} in the Appendix A non-empty, by Lemma \ref{l:subset} and Lemma \ref{l:ii}, there exists some $y_0 \in \mathbb{R}$ such that $v^{y_0}_0 \in \mathcal{E} \cap \bar{\omega}(u_0)$, thus $v^{y_0}_0 \in \omega(u_0)$. Now by (C1), for each $\delta > 0$ there exists a sufficiently large $k_0 \in \mathbb{N}$ such that $v^{y_0-\delta}_0 \leq T^{k_0}(u_0) \leq v^{y_0+\delta}_0$. As all $v^y_0$ are $T$-invariant, by the maximum principle we have that for all $k \geq k_0$, $v^{y_0-\delta}_0 \leq T^k(u_0) \leq v^{y_0+\delta}_0$, thus $\omega(u_0)$ contains only $v^{y_0}_0$.
\end{proof}

Denote below by $z(t)$ the solution of (\ref{r:main}), $z(0)=z_0$.

\begin{lemma} \label{l:transversal}
	Assume in the extended case that $\mu_0$ satisfies (N1). Then there exists a set $\mathcal{U}$ of full measure such that for $u_0 \in \mathcal{U}$ and for any $z_0 \in \omega(u_0)$, $z(t)-v^y(t)$ can not have a multiple zero for any $y,x,t \in \mathbb{R}$.
\end{lemma}

\begin{proof}
	By Proposition \ref{p:transversal}, for a given $y \in \mathbb{R}$, there exists a set of full measure $\mathcal{U}_y$ such that if $u_0 \in \mathcal{U}_y$ and $z_0 \in \omega(u_0)$, $z(t)-v^y(t)$ can not have a multiple zero for any $x,t \in \mathbb{R}$. Now the set $\mathcal{U}=\cap_{y \in \mathbb{Q}} \mathcal{U}_y$ also satisfies $\mu(\mathcal{U})=1$. Assume there is $u_0 \in \mathcal{U}$ such that for some $z_0 \in \omega(u)$ and some $y_0 \in \mathbb{R}$, $z(t)-v^{y_0}(t)$ have a multiple zero for some $t,x \in \mathbb{R}$. However, by an analogous argument as in Lemma \ref{l:3_3}, we can find $\delta > 0$ such that for each $y \in (y_0-\delta,y_0+\delta)$, there exists $\tilde{t}$ in a neighbourhood of $t$ such that $z_0(\tilde{t})-v^y(\tilde{t})$ has a multiple zero, which is impossible for rational $y$, thus a contradiction.
\end{proof}

\begin{lemma} \label{l:asymptotics}
	Assume in the extended case that $\mu_0$ satisfies (N1). For $\mu_0$-a.e. $u_0$, we have that $\omega(u_0)\subset \mathcal{V} \cup \mathcal{H}$.
\end{lemma}

\begin{proof}
	To show (i), we show analogously as in the proof of Lemma \ref{l:ii} in the extended case, by applying Lemma \ref{l:transversal}, that there exists a set of full measure $\mathcal{U}$ so that for any $u_0 \in \mathcal{U}$ and any $z_0 \in \omega(u)$, $z-v^y_0$ can not have a multiple zero for any $y \in \mathbb{R}$. Analogously as in the same proof, we obtain $\omega(u_0) \subset \mathcal{V} \cup \mathcal{H}$ (the possibility that $z_0 \in \mathcal{H}$ can not be a-priori eliminated).
\end{proof}

\begin{lemma} \label{l:weak*}
	Assume in the extended case that $\mu_0$ satisfies (N1). Then $\omega$-limit set of $\mu_0$ in the weak$^*$-topology consists of measures supported on $\mathcal{V}$.
\end{lemma}	

\begin{proof}
	It is a standard ergodic theoretical fact that if $\nu_0 \in \omega(\mu_0)$ ($\omega$-limit set with respect to the weak$^*$ topology of iterations of $\mu_0$ induced by $T$), then $\supp \nu_0 \subset \cup_{u_0 \in \supp \mu_0} \omega(u_0)$ \cite{Walters}, thus $\nu_0$ is by Lemma \ref{l:asymptotics} supported on $\mathcal{V} \cup \mathcal{H}$. It is easy to see that $\nu_0$ must be $S$-invariant, thus as no $h \in \mathcal{H}$ is $S$-recurrent, by the Poincar\'{e} recurrence theorem, $\nu_0(\mathcal{H})=0$. Now $\nu_0(\mathcal{V})=1$, and as $\mathcal{V}$ is closed, $\supp \nu_0 \subset \mathcal{V}$.
\end{proof}

\begin{proof}[Proof of Proposition \ref{p:C1}]
	Theorem \ref{t:main2}, (i) is a restated Lemma \ref{l:C1}; (ii) is Lemma \ref{l:ii}, and (iii) can easily be deduced from (ii), as then $\mathcal{B}_y \cap \mathcal{E}=\lbrace v^y_0 \rbrace$. Corollary \ref{c:bounded} follows directly from Lemma \ref{l:bbounded}, where we obtain the required $y_0$ directly from (B3). Corollary \ref{c:asymptotics} follows from Lemmas \ref{l:asymptotics} and \ref{l:weak*}.
\end{proof}

\section{Proof of Corollary \ref{c:weak*}} \label{s:short}

We now give a more general condition which suffices instead of (N2) to establish Corollary \ref{c:weak*}.
\begin{itemize}
	\item[(C2)] Assume in the extended case that $\mu_0$ is a $S$-invariant Borel-probability measure with $y_0 = \int \int_0^1u_0(x) \, dx \, d\mu_0(u_0)$, such that for $\mu_0$-a.e. $u_0$, and for each $w_0 \in \omega(u_0)$,
	$$
	\lim_{x \rightarrow \infty} \frac{1}{x}\int_0^x w_0(z)dz = \lim_{x \rightarrow \infty} \frac{1}{x}\int_{-x}^0 w_0(z)dz=y_0.
	$$
\end{itemize}

We now have, while assuming (A1-3), (B1-3) and (N1) in the E/P case:

\begin{lemma} \label{l:last}
	(i) (N2) implies (C2),
	
	(ii) (C2) implies Corollary \ref{c:weak*}.
\end{lemma}

\begin{proof}
	(i) is straightforward. It is also easy to check that (C2) eliminates all the possibilities in Corollary \ref{c:asymptotics} except those remaining in Corollary \ref{c:weak*}.
\end{proof}

\part{Examples and open problems}

\section{The generalized Burgers equation} \label{burgers}

We show here that our results apply to the family of nonlinearities generalizing the nonlinearity in (\ref{r:burgers}):
\begin{equation*}
g(t,x,u,u_x)=-h(u)u_x+\hat{g}(t,x),
\end{equation*}
where $h,\hat{g}$ are continuous, $\hat{g}$ is locally H\"{o}lder continuous in $t$ and $1$-periodic in $x$, $t$, such that for all $t \in \mathbb{R}$, $\int_0^1 \hat{g}(t,x)dx = 0$; and $h$ is locally Lipschitz continuous. 

It suffices to show that all our assumptions hold:

\begin{proposition} \label{l:burgershigh}
	The equation (\ref{r:main}) with the nonlinearity $g$ as above satisfies (A1-3) and (B1-3).
\end{proposition}

\begin{proof} It is straightforward to check (A1-3). To show (B1-3), it suffices to consider (\ref{r:main}) in the bounded case, for $u \in \mathcal{X}^{\alpha}=H^{2 \alpha}(\mathbb{S}^1)$. First note that for any continuous $\hat{h} : \mathbb{R} \rightarrow \mathbb{R}$, by considering $H(y)=\int_0^1 \hat{h}(z)dz$, thus $dH(u)/dx=\hat{h}(u)u_x$, we get
	\begin{equation}
	\int_0^1 \hat{h}(u)u_x dx = 0. \label{r:partint}
	\end{equation}
	Now (B1) is self-evident, and (B3) follows easily by differentiating $\int_0^1 u(x)dx$ with respect to $t$ and using (\ref{r:partint}) with $\hat{h}=h$ and partial integration. To show (B2), let $c_0 = \max_{x,t \in [0,1]}|\hat{g}(x,t)|$, fix $y \in \mathbb{R}$ and choose $u \in \mathcal{X}^{\alpha}$ such that $\int_0^1u(x)dx=y$. Let $t_0 \in \mathbb{R}$, and assume the solution of (\ref{r:main}), $u(t_0)=u$ exists on $[t_0,t_1)$. We differentiate for an integer $p \geq 1$:
	\begin{align}
	\frac{d}{dt}\frac{1}{2p}\int_0^1 (u(x)-y)^{2p}dx & = \int_0^1 (u(x)-y)^{2p-1}u_{xx} dx - \int_0^1 (u(x)-y)^{2p-1}h(u)u_x dx  \notag \\ & \hspace{4ex} + \int_0^1 (u(x)-y)^{2p-1}\hat{g}(t,x)dx \notag \\
	& = - (2p-1)\int_0^1 (u(x)-y)^{2p-2}u_x^2(x)dx + \int_0^1 (u(x)-y)^{2p-1}\hat{g}(t,x)dx, \label{r:burg1}
	\end{align}
	where in the second row we partially integrated the first term and used that $u(x)$ is 1-periodic, and also (\ref{r:partint}) with $\hat{h}(u)=(u-y)^{2p-1}h(u)$ applied to the second term. 
	
	As $w(x):=(u(x)-y)^p$, $w \in C^1(\mathbb{S}^1)$  has a zero for some $x \in \mathbb{S}^1$, we can apply the $L^2$-Poincar\'{e} inequality to $w$ to obtain
	\begin{equation}
	\int_0^1 (u(x)-y)^{2p}dx \leq \frac{p^2}{\pi^2}\int_0^1 (u(x)-y)^{2p-2}u_x^2(x)dx. \label{r:burg2}
	\end{equation}
	By the weighted Young's inequality applied to the integrand in the last term in (\ref{r:burg1}), we get
	\begin{equation}
	(u(x)-y)^{2p-1}\hat{g}(t,x) \leq \frac{(2p-1)\pi}{2 p^2}(u(x)-y)^{2p} + \frac{1}{2\pi} c_0^{2p}. \label{r:burg3}
	\end{equation}
	Inserting (\ref{r:burg2}) and (\ref{r:burg3}) into (\ref{r:burg1}), we now have
	\begin{equation*}
	\frac{d}{dt}\frac{1}{2p}\int_0^1 (u(x)-y)^{2p}dx  \leq - \frac{(2p-1)\pi}{2 p^2}\int_0^1 (u(x)-y)^{2p}dx + \frac{1}{2\pi} c_0^{2p},
	\end{equation*}
	thus by the Gronwall inequality, the following implication holds:
	\begin{equation}
	||u(t_0)-y||_{L^{2p}(\mathbb{S}^1)} \leq c_{2p} \Rightarrow ||u(t)-y||_{L^{2p}(\mathbb{S}^1)} \leq c_{2p}, \: t \in [t_0,t_1),
	\end{equation}
	where $$c_{2p}=\left( \frac{p^2 }{(2p-1)\pi^2}\right)^{\frac{1}{2p}}c_0.$$ 
	Now if $||u(t_0)-y ||_{L^{\infty}(\mathbb{S}^1)} \leq c_0$, we have that for all integer $p \geq 2\pi^2$ and all $t \in [t_0,t_1)$, $||u(t)-y||_{L^{2p}(\mathbb{S}^1)} \leq c_{2p}$. As $\lim_{p \rightarrow \infty}c_{2p}=c_0$, we conclude that for all $t \in [t_0,t_1)$, $||u(t)-y||_{L^{\infty}(\mathbb{S}^1)} \leq c_0$, thus (C1) holds with $d(y):=c_0$, where (C1),(ii) follows from (\ref{r:B2}).
\end{proof}

\begin{example} \label{e:burgN2} We now consider the Burgers equation (\ref{r:burgers}) and show that any measure $\mu_0$ satisfying Sinai's assumptions \cite{Sinai:91} satisfies in particular (N2), and relate assumptions from \cite{Sinai:91} to our setting.
	
Consider the Cole-Hopf substitution $u=-2\varphi_x/\varphi$, $\varphi(x)=\exp\left( -\frac{1}{2}\int_0^xu(y)dy \right)$, as in \cite{Sinai:91,Sinai:96,Sinai:98}, and get for $\varphi$ the linear equation
	\begin{equation}
	\varphi_t = \varphi_{xx} - \frac{1}{2}\hat{g}(x,t)\varphi. \label{r:lin}
	\end{equation}
	Let $\varphi^y_0$ be the transformed family $v^y_0$, $y \in \mathbb{R}$, as in (C1), which exists by Lemmas \ref{l:C1} and \ref{l:burgershigh}. By definition, $\varphi^0_0$ is non-negative, continuous and $1$-periodic in $x,t$, thus $\varphi^0(t)$ is uniformly bounded in $L^{\infty}(\mathbb{R})$. Assume now $\nu_0$ is a $S$-invariant measure supported on a family of sufficiently smooth functions $f_0$, such that for some $0 < c_1 < c_2$, for $\nu_0$-a.e. $f_0$ we have that $c_1 \leq f_0 \leq c_2$. Then we can find $0 < c_3 < 1$ such that $c_3 \varphi^0_0 \leq f_0 \leq \frac{1}{c_3}\varphi^0_0$, thus by the maximum principle and linearity of (\ref{r:lin}), $f(t)$ is bounded uniformly in $t$ in $L^{\infty}(\mathbb{R})$. It is easy to check that then the measure $\mu_0$ which is the push of  $\nu_0$ with respect to the Cole-Hopf substitution satisfies (N2) with $y_0=0$.
	
	The assumptions of Sinai in \cite{Sinai:91} on the probability measure can be understood as analogous to ours, as his Assumption 2 (the spatial invariance of expectation) is somewhat weaker form of $S$-invariance, his Assumption 1 when combined with the maximum principle as above implies (N2), and the Assumption 3 seems to be related to the finiteness of density of zeroes in (N1), yet to be understood.
\end{example}

\section{Further examples} \label{s:examples}

\subsection{The B/A case} \label{ss:BDC} In the B/A case, $\mathcal{E}$ is equal to the closure of the set of equilibria and periodic orbits in $\mathcal{B}$. This follows from the Poincar\'{e}-Bendixson theorem \cite{Fiedler89} and Lemma \ref{l:subset}, as by \cite[Theorem 1]{Fiedler89}, the only recurrent orbits in the B/A case are equilibria and periodic orbits.

Theorem \ref{t:main1} in the B/A case can be deduced from results in \cite{Fiedler89}, Theorems 1, 2 and Lemma 3.3.

\vspace{2ex}

\subsection{Embedded vector fields in the B/P case} \label{ss:embed} Consider planar vector fields constructed by Fiedler and Sandstede \cite{Fiedler:92}, embedded in the B/P case of (\ref{r:main}). Then the union of supports of invariant measures of these vector fields are mapped into a subset of $\mathcal{E}$. This complements well Theorem \ref{t:main1}, in the sense that $\mathcal{E}$ can have arbitrary complexity of a 2d vector field. In particular, one can embed in $\mathcal{E}$ invariant measures with positive metric entropy with respect to $T$.

\subsection{Extended gradient systems} \label{ss:extended} Consider $g = - \partial V(x,u)/\partial u$, with a $C^2$ $V$, $1$-periodic in $x$, bounded from below, in the extended case. Then $g$ satisfies (A1-3), and is an example of an {\it extended gradient system}, introduced in \cite{Gallay:01}. Under an additional assumption that for any $u(0) \in H^{2\alpha}_{\text{ul}}(\mathbb{R})$, $3/4 < \alpha < 1$, the solution exists for all $t\geq 0$ and is uniformly bounded in $\mathcal{X}^{\alpha}:=H^{2\alpha}_{\text{ul}}(\mathbb{R})$ (see \cite{Gallay:01},\cite{Gallay:12} for further details), we can establish the following:

\begin{thm} \label{t:extgard}
	(i) The ergodic attractor consists of equilibria, i.e. it is given with $\mathcal{E}=\lbrace u \in H^2_{\text{ul}}(\mathbb{R}), u_{xx}=\partial V(x,u)/\partial u \rbrace$, and $\pi : \mathcal{E} \rightarrow \mathbb{R}^2$ given with (\ref{d:pi}) is one-to-one.
	
	(ii) For all $u \in  H^{2\alpha}_{\text{ul}}(\mathbb{R})$, $3/4 < \alpha \leq 1$, we have that $\bar{\omega}(u) \subset \mathcal{E}$.
	
	(iii) Given any $S$-invariant measure $\mu$ on  $H^{2\alpha}_{\text{ul}}(\mathbb{R})$, for $\mu$-a.e. $u \in H^{2\alpha}_{\text{ul}}(\mathbb{R})$, we have that $\omega(u) \subset \mathcal{E}$.
	
	(iv) Given any $S$-invariant measure $\mu$ on $H^{2\alpha}_{\text{ul}}(\mathbb{R})$, its $\omega$-limit set in the weak$^*$ topology of the induced semiflow on the space of measures consists of measures supported on $\mathcal{E}$.
\end{thm}

The claims (i),(iii),(iv) were proved in \cite{Slijepcevic:99,Slijepcevic00} (the fact that $\pi$ is one-to-one follows from uniqueness of the solutions of the ordinary differential equation in the description of $\mathcal{E}$), and (ii) was shown in \cite{Gallay:01}, \cite{Gallay:12}.

Theorem \ref{t:extgard}, (i) is an example of a family for which Theorem \ref{t:main1b} holds without a non-degeneracy restriction; (ii) strengthens in this particular case the properties of the ergodic attractor in the extended case from Subsection \ref{ss:ergextended}; and (iii), (iv) give an example of another family of nonlinearities $g$ for which the claims in Corollaries \ref{c:asymptotics} and \ref{c:weak*} hold. The main tool in the proof of (i),(iii),(iv) is the following Lyapunov function on the space of $S$-invariant measures on $\mathcal{X}^{\alpha}$:
$$
L(\mu)=\int_{\mathcal{X}^{\alpha}} \int_0^1 \left( \frac{u_x^2(x)}{2}+\frac{\partial V(x,u)}{\partial u} \right) dx \: d\mu(u),
$$
which plays an analogous role to the zero function in this paper.

\subsection{The Allen-Cahn equation} \label{ss:AllenCahn} We give an example why (N2) is required to obtain sharper conclusions (ii) in the claims of Corollaries \ref{c:asymptotics} and \ref{c:weak*}. Even though it does not strictly belong to the class of Burgers-like equations, we believe it is illustrative.

\begin{example} Consider the nonlinearity as in Subsection \ref{ss:extended}, with $V=\frac{1}{4}u^4 -\frac{1}{2}u^2$, thus $g = u - u^3$. As done in \cite{Polacik:15}, the phase-plane analysis of the family of equilibria and Theorem \ref{t:extgard} show that $\mathcal{E}$ consists of the following equilibria: $u^-\equiv -1$, $u^+ \equiv 1$, a two families of spatially heteroclinic functions $h^+_y$, $h^-_y$, such that $\lim_{x \rightarrow -\infty}h^-_y(x)=\lim_{x \rightarrow \infty}h^+_y(x)=1$, $\lim_{x \rightarrow -\infty}h^+_y(x)=\lim_{x \rightarrow \infty}h^-_y(x)=-1$, characterized by $h^+_y(y)=h^-_y(y)=0$, and further spatially periodic functions with various periods and values in $(-1,1)$. 
	
Similarly as in \cite{Polacik:15}, consider a smooth profile $v^0 : [-n,n] \rightarrow \mathbb{R}$, $-1 < v^0 \leq 0$, $v^0(-n)=v^0(n)=0$, such that $\frac{1}{2n}\int_{-n}^n v^0(x)dx\leq -1 + \delta$ for $\delta > 0$ small enough, let $v^1 = - v^0$, and embed the Bernoulli measure as in Example \ref{e:nondeg} such that to each sequence $(\omega_k)_{k \in \mathbb{Z}}$ we associate a function $u$ by combining profiles $v^{\omega(k)}$ to obtain a $S^{2n}$-invariant measure. We easily obtain a $S$-invariant measure $\mu$ by taking $2n$ copies of its translates. Pol\'{a}\v{c}i{k} \cite{Polacik:15} has shown that there exists $u \in \supp \mu$ such that $\omega(u)$ contains orbits not in $\mathcal{E}$. However, one can show by applying Theorem \ref{t:extgard}, (iii) and techniques from \cite{Polacik:15}, that for $n$ large enough and $\delta>0$ small enough, for $\mu$-a.e. $u$, $\omega(u) = \lbrace u^+,u^-,h^+_y, y \in \mathbb{R}, h^-_y, y \in \mathbb{R} \rbrace$, thus spatially heteroclinic functions in $\omega$-limit sets in the sense of Corollary \ref{c:asymptotics},(i) can not be avoided in general. We also obtain that for $\mu$-a.e. $u$, $\bar{\omega}(u)=\lbrace u^+, u^- \rbrace$, and that the $\omega$-limit set of $\mu$ in the weak$^*$-topology consists of a single measure $\frac{1}{2}\delta_{u^-}+\frac{1}{2}\delta_{u^+}$. This shows that the $\omega$-limit measure in the sense of Corollary \ref{c:weak*} is not necessarily supported on a single function.
\end{example}

\section{Open problems} \label{s:open}

\subsection{Non-degeneracy of measures} \label{ss:nondeg} we propose two approaches to further characterize and possibly remove the non-degeneracy restrictions to the results in the extended case. First, the following general ergodic-theoretical conjecture (a generalization of Proposition \ref{l:ergodic}) would imply Theorem \ref{t:main1b} without a non-degeneracy restriction:

\begin{conjecture} \label{con:one}
Assume $(\Omega,\mathcal{F},\nu)$ is a probability space, and that $\hat{\sigma}, \hat{\tau}: \Omega \rightarrow \Omega$ are commuting, measurable, $\nu$-invariant maps. Assume that $\varphi, \zeta, \delta : \Omega \rightarrow \mathbb{R}$ are measurable, that $\zeta, \delta \geq 0$ and that $\nu$-a.e.,
\begin{equation}
\varphi \circ \hat{\sigma} -\varphi  + \zeta \circ \hat{\tau} - \zeta \geq \delta . 
\end{equation}
Then $\delta = 0$, $\nu$-a.e..
\end{conjecture}

\begin{problem}[1]
Prove, or disprove Conjecture \ref{con:one}.
\end{problem}

An alternative approach is to characterize non-linearities $g$ and invariant sets for which all the $S$-invariant measures are non-degenerate. Let $\pi_1 : C(\mathbb{R}) \rightarrow  C ([0,1])$, $\pi_1(u)=u|_{[0,1]}$, let $3/4 < \alpha < \gamma <1$ and let $\mathcal{Y}:=H^{2\gamma}_{\text{ul}} (\mathbb{R}) \cap \hat{T}(0,-\delta_0)H^{2\alpha}_{\text{ul}} (\mathbb{R})$ for some $\delta_0 >0$. For example, we have the following:

\begin{lemma} \label{l:sufficient} Assume $\mu_1$, $\mu_2$ are $S$-invariant measures supported on a subset of $\mathcal{Y}$ bounded in $H^{2\gamma}_{\text{ul}} (\mathbb{R})$, such that
	\begin{equation}
	\sup \frac{||\pi_1(u)-\pi_1(v)||_{H^{2\gamma} ([0,1])}}{||\pi_1(u)-\pi_1(v)||_{H^{2\alpha} ([0,1])}} < \infty, \label{r:sufficient}
	\end{equation}
	where supremum goes over $u \in \supp \mu_1, v \in \supp \mu_2, u \neq v$. Then we have that for any such $u,v$, $\hat{\zeta}(u,v) < \infty$. Furthermore, $\mu_1 \times \mu_2$ is non-degenerate.
\end{lemma}

\begin{remark} \label{rr:sufficient}
For example, this holds if $\mu_1$, $\mu_2$ are supported on disjoint sets in $\mathcal{Y}$, bounded in $H^{2\gamma}_{\text{ul}} (\mathbb{R})$; or alternatively if they are supported on finite sets in $\mathcal{Y}$.
\end{remark}

\begin{proof}
By assumptions, the set 
$$
\mathcal{C}:= \left\lbrace \frac{\pi_1(u)-\pi_1(v)}{||\pi_1(u)-\pi_1(v)||_{H^{2\alpha} ([0,1])}}, \: u \in \supp \mu_1, v \in \supp \mu_2, u \neq v \right\rbrace.
$$
is compact in ${H^{2\alpha} ([0,1])}$. By the local structure of zeroes and the fact that for all $w=u-v$, the solution exists backward in time on the interval $(-\delta_0,0]$, we can find an open, and by compactness finite cover $\mathcal{U}_j$ of $\mathcal{C}$, $j=1,...,m$, such that $z(w)$ is uniformly bounded for $w \in \mathcal{U}_j$. This and $S$-invariance of $\mu$ implies a finite uniform bound on $z(S^nu,S^nv)$, for $u \in \supp \mu_1$, $v \in \supp \mu_2$, $n \in \mathbb{Z}$.
\end{proof}

\begin{example} \label{e:itisok}
	The ergodic attractor for nonlinearities from subsection (\ref{ss:extended}) is non-degenerate. Indeed, consider a $S,T(t)$-invariant measure $\mu$ supported on a set $\tilde{\mathcal{B}}$ bounded in $H^{2\gamma}_{\text{ul}}(\mathbb{R})$, thus bounded in $L^{\infty}(\mathbb{R})$ by a constant $c_1 > 0$, and let $u,v \in \supp \mu$. By Theorem \ref{t:extgard}, (i), $u_{xx}=\partial V(x,u)/\partial u$, $v_{xx}=\partial V(x,v)/\partial v$, thus by the Mean Value Theorem,
	\begin{align*}
	||(\pi_1(u)-\pi_1(v))_{xx}||_{L^1([0,1])} \leq \max_{ x \in [0,1], |\xi| \leq c_1 }\left| \frac{\partial^2 V(x,\xi)}{\partial \xi^2} \right| ||\pi_1(u)-\pi_1(v)||_{L^1([0,1])}. 
	\end{align*}
	We can now deduce (\ref{r:sufficient}) by applying the standard interpolation and embedding estimates.	
\end{example}
Now it would suffice to answer the following:

\begin{problem}[2]
	Characterize nonlinearities $g$ such that for any $z = z(0) \in H^{2\alpha}_{\text{ul}} (\mathbb{R})$, there exists $t > 0$ and an invariant set $\tilde{\mathcal{B}}$, bounded in $H^{2\gamma}_{\text{ul}} (\mathbb{R})$, such that any $u,v \in \tilde{\mathcal{B}}$, $u \neq v$, satisfy (\ref{r:sufficient}), and such that $z(t) \in \tilde{\mathcal{B}}$.
\end{problem}

\begin{problem}[3]
	Characterize nonlinearities $g$ such that the attractor $\mathcal{A}$ (i.e. the set of the entire solutions) in the extended case consists of $u,v$, $u\neq v$ satisfying (\ref{r:sufficient}).
\end{problem}

\subsection{Further extended gradient systems} \label{ss:further}

As noted by Zelenyak \cite{Zelenyak:97}, and extended by Matano, Fiedler, Pol\'{a}\v{c}ik, Rocha and others (\cite{Fiedler:12}, \cite{Polacik:02} and references therein), there is a number of examples of nonlinearities $g$ with a Lyapunov function on the bounded domain (with periodic or other boundary conditions). The discussion in Subsection \ref{ss:extended} thus naturally leads to the following:

\begin{problem}[4] Prove (or disprove) that for all nonlinearities $g$ for which there exists a Lyapunov function in the bounded case (i.e. with periodic boundary conditions), the conclusions (i)-(iv) of Theorem \ref{t:extgard} hold.	
\end{problem}

For example, one can show that it holds for the cases considered in \cite{Fiedler:12}.

\subsection{Related problems}

We believe the application of the zero function on the space of measures could be applied to other classes of dynamical systems, and systems with a random force:

\begin{problem}[5]
	Extend results for the Burgers like equations to the quasi-periodic force case considered in \cite{Sinai:98}.
\end{problem}

\begin{problem}[6] {\it Work in progress.}
	Extend results for the Burgers like equations to the random force case considered in \cite{E:00, Sinai:96}, by using the fact that for the difference of two weak solutions $u(t)$, $v(t)$ with the same random force, the random force cancels out and the difference is smooth enough to apply the zero function method.
\end{problem}

\begin{problem}[7]
 Investigate whether the results for the equations $$u_t = \varepsilon u_{xx} + g(t,x,u,u_x),$$ $g$ a  Burgers like nonlinearity, extend to the entropy solutions in the inviscid limit $\varepsilon \rightarrow 0$, as considered in \cite{E:00}, by e.g. using in addition the zero function techniques for perturbations of parabolic differential equations developed by Pol\'{a}\v{c}ik and Tere\v{s}\v{c}ak \cite{Terescak:94}, or another method.
\end{problem}

\begin{problem}[8]
	Consider all the problems in this paper and apply zero-function techniques on the space of measures for analogous 1d, order-preserving discrete-space, continuous-time problems without and with a random force (the Frenkel-Kontorova models, \cite{Baesens:05,Slijepcevic13a} and references therein), or order-preserving discrete-space, discrete-time models (monotone coupled map lattices and probabilistic cellular automata, \cite{Coutinho:05,Toom:90} and references therein).
\end{problem}

This program has already been initiated in the case of the Frenkel-Kontorova models \cite{Slijepcevic13b,Slijepcevic15}.

\part{Appendices}

\appendix

\section{Fractional uniformly local spaces and their topologies} \label{s:fractional} 

We recall the key facts on uniformly local spaces used throughout the paper in the extended case. Let $\varphi^y(u)(x)=u(x+y)$ be the translation, $y \in \mathbb{R}$. The uniformly local spaces are given with:
\begin{align*}
||u||_{L^2_{\text{ul}}(\mathbb{R})}&= \sup_{y \in \mathbb{R}} \left( \int_{\mathbb{R}} e^{-|x+y|}u(x)^2 dx \right) ^{1/2},  \\
L^2_{\text{ul}}(\mathbb{R}) &= \left\lbrace  u \in L^2_{\text{loc}}(\mathbb{R}) , 
\: ||u||_{L^2_{\text{ul}}(\mathbb{R})} < \infty, \: 
\lim_{y \rightarrow 0}||\varphi^yu-u||_{L^2_{\text{ul}}(\mathbb{R})}=0  \right\rbrace ,  \\
H^k_{\text{ul}}(\mathbb{R}) &= \left\lbrace u\in L^2_{\text{ul}}(\mathbb{R}) \: | \: \: \partial^j_tu \in L^2_{\text{ul}}(\mathbb{R}) \text{ for all } j\leq k \right\rbrace.
\end{align*}
It is straightforward to show that the unbounded linear operator $Aq=-u_{tt}$ on $\mathcal{X}:=L^2_{\text{ul}}(\mathbb{R})$ has the domain $D(A)=H^{2}_{\text{ul}}(\mathbb{R})$, and that by using an explicit expression of the heat kernel, $A$ generates an analytic semigroup $\exp(-tA)$ on $\mathcal{X}$ with the usual a-priori bounds, thus it is sectorial \cite[Section 3]{Henry}. We can thus set $A_1=A+I$, and then $\sigma(A_1) \geq 1 > 0$, and define the fractional powers $A_1^{\alpha}$, $0 < \alpha < 1$, and the space $\mathcal{X}^{\alpha}:=D(A_1^{\alpha})$ as in \cite[Section 1.4]{Henry}. We occasionally write $H^{2 \alpha}_{\text{ul}}(\mathbb{R})$ instead of $\mathcal{X}^{\alpha}$ to distinguish it from the bounded case. We always use the graph norm on $\mathcal{X}^{\alpha}$
$$
||u||_{\mathcal{X}^{\alpha}}:= ||A_1^{\alpha}u||_{L^2_{\text{ul}}(\mathbb{R})}.
$$
Now local existence, regularity and continuity with respect to initial conditions of (\ref{r:main}) holds on $\mathcal{X}^{\alpha}$, with the usual definitions of the mild solution; and the variations of constants formula holds (see \cite{Gallay:01}, Section 7.2 for details). 

The following elementary observation will imply that many choices of "local", coarse topologies on $\mathcal{X}^{\alpha}$ are the same

\begin{lemma} \label{l:topology}
	Consider a subset $\mathcal{Z} \subset \mathcal{Y}_0 \subset \mathcal{Y}_1$, where $\mathcal{Y}_0$ and $\mathcal{Y}_1$ are metrizable and complete topological spaces with respective topologies $\tau_0$, $\tau_1$, such that $\tau_1|_{\mathcal{Y}_0} \subset \tau_0$. 
	Furthermore, assume that $\mathcal{Z}$ is relatively compact in both $ \mathcal{Y}_0$ and $\mathcal{Y}_1$. Then the closure of $\mathcal{Z}$ in $\mathcal{Y}_0$ and $\mathcal{Y}_1$ is the same, and the topologies $\tau_0$ and $\tau_1$ induced on $\operatorname{Cl}(\mathcal{Z})$ are the same.
\end{lemma}

\noindent (We denote by $\tau_1|_{\mathcal{Y}_0}$ the induced topology $\tau_1$ on $\mathcal{Y}_0$).
\begin{proof}
	We first show that a sequence $u_n \in \mathcal{Z}$ converges in $\mathcal{Y}_0$ if and only if it converges in $\mathcal{Y}_1$. As $\tau_1|_{\mathcal{Y}_0} \subset \tau_0$, the non-trivial direction is that convergence in $\mathcal{Y}_1$ implies convergence in $\mathcal{Y}_0$. Assume the contrary, and let $v \in \mathcal{Y}_1$ be a limit of $u_n \in \mathcal{Z}$ in $\mathcal{Y}_1$. By the assumptions and relative compactness, there exists a subsequence $u_{n_k}$ converging to some $\tilde{v} \in \mathcal{Y}_0$ in $\mathcal{Y}_0$. But then $u_{n_k}$ also converges to $\tilde{v}$ in $\mathcal{Y}_1$, thus $\tilde{v}=v$. 
	We deduce that the closure of $\mathcal{Z}$  in both topologies is the same. As the topology in metric spaces is entirely determined by convergence, it suffices to repeat the argument above for any sequence in $u_n \in Cl({\mathcal{Z}})$.
\end{proof}

We apply Lemma \ref{l:topology} by setting $\mathcal{Y}_0$ to be the closure of $\mathcal{B}$ in  $H^{2\alpha}_{\text{loc}}(\mathbb{R})$ with the induced $H^{2\alpha}_{\text{loc}}(\mathbb{R})$ topology denoted by $\tau_0$, and $\mathcal{Z}$ to be the subset of $\mathcal{B}$ of all $u_0$ such that the solution of (\ref{r:main}), $u(0)=u_0$ can be extended backwards in time on $(\delta_0,0)$ ($\delta_0$ is as in Section \ref{s:prelim}). Let $\mathcal{Y}_1$ be the closure of $\mathcal{B}$ in any of the following spaces with respective induced topologies $\tau_1$: $L^{\infty}_{\text{loc}}(\mathbb{R})$ (i.e. we consider uniform convergence on compact sets); $C^1_{\text{loc}}(\mathbb{R})$; or $H^{2\delta}_{\text{loc}}(\mathbb{R})$ for $1/2 \leq \delta \leq \alpha$ (defined as topology of convergence in $H^{2\gamma}([-n,n])$ for all $n \in \mathbb{N}$). In other words, all the aforementioned topologies on $\tilde{ \mathcal{B}}=\operatorname{Cl}(\mathcal{Z})$ are equivalent, and $\tilde{ \mathcal{B}}$ is compact.

\section{Interpretation of the ergodic attractor} \label{s:ergodic}

This Appendix is dedicated to showing elementary properties of the ergodic attractor $\mathcal{E}$, and its interpretation in the bounded and the extended case. We use all the notation and assumptions from Section \ref{s:prelim}, though the results hold for general continuous semiflows $T(t)$, respectively maps $T$ on compact metric spaces $\tilde{ \mathcal{B}}$. In the extended case, we in addition use that $T(t)$, resp. $T$ commute with a homeomorphism $S$.

\subsection{Ergodic attractor}
We first show that $\mathcal{M}(\mathcal{B})$, thus $\mathcal{E} = \cup_{\mu \in \mathcal{M}(B)}\: \supp \mu$ are non-empty.
\begin{lemma} \label{l:nonempty}
	The set $\mathcal{M}(\mathcal{B})$ is non-empty.
\end{lemma}

\begin{proof}
	Existence of the invariant measure in the bounded case is the classical Krylov - Bogolioubov theorem for continuous maps, resp. continuous semiflows on compact metrizable sets \cite{Walters}. To show it in the E/P case, i.e. for commuting $S,T$, it suffices to find a weak$^*$-convergent subsequence of the sequence of Borel probability measures
	$$ \sum_{m=-L}^{L-1} \sum_{n=1}^L \frac{1}{2L^2} (S^m)^*(T^n)^*\delta_{u},$$
	where $f^*\mu$ is the standard pull of a measure, and $\delta_u$ is the (Dirac) probability measure concentrated on a single, fixed $u \in \mathcal{A}$. Analogously we prove the claim in the E/A case by replacing summing $T^n$ from $0$ to $L$ with integrating $T(t)$ from $0$ to $L$.
\end{proof}

\begin{lemma} \label{l:properties}
	(i) $\mathcal{E}$ is invariant.
	
	(ii) $\mathcal{E}$ is closed and compact,
	
	(iii) $\mathcal{E}\subset \mathcal{A}$, thus it consists of entire orbits. 
	
	(iv) In the autonomous case $T(t)|_{\mathcal{E}}$ is a continuous flow. In the periodic case, $T|_{\mathcal{E}}$, $S|_{\mathcal{E}}$ are homeomorphisms.
\end{lemma}

\begin{proof} The claim (i) follows from the invariance of every $\mu \in \mathcal{M}(\mathcal{B})$ and the definition of support. To show (ii), consider a convergent sequence $u_n \in \mathcal{E} \subset \mathcal{A}$ converging to some $u \in \mathcal{A}$, and the associated invariant measures $\mu_n \in \mathcal{M}(\mathcal{B})$ such that $u_n \in \supp \mu_n$. The measure 
$$ \mu = \sum_{n=1}^{\infty} 2^{-n}\mu_n$$
is by definition in $\mathcal{M}(\mathcal{B})$, and also by the definition the support of $\mu$ it contains supports of $\mu_n$ for all $n \geq 1$, thus $u_n \in \supp\mu$. As $\supp\mu$ is by definition closed, we deduce that $u \in \supp\mu$, thus $\mathcal{E}$ is closed. As it is a subset of a compact set $\tilde{ \mathcal{B}}$, it is compact. 

Let $\nu \in \mathcal{M}(\mathcal{B})$. Consider the sequence of sets $\mathcal{B}_n :=T^n(\mathcal{B})$. As by (A4), $T(\mathcal{B}) \subset \mathcal{B}$, the sequence $\mathcal{B}_n$ is decreasing. By the characterization of the attractor as the set of entire orbits, $\mathcal{A}=\cap_{k=1}^{\infty} \mathcal{B}_k$. As for $n \geq 1$, $\mathcal{B}_n \subset \tilde{\mathcal{B}}$ is compact, and by $T$-invariance of $\mu$, all of $\mathcal{B}_n$ are of full measure, thus $\cap_{k=1}^{\infty} \mathcal{B}_k$ is of full measure and closed. We conclude that $\supp\nu \subset \cap_{k=1}^{\infty} \mathcal{B}_k= \mathcal{A}$, which completes (iii). The claim (iv) follows from (i),(iii) and the properties of $T(t)$, $T$ and $S$ on $\mathcal{B}$.
\end{proof}

\subsection{Asymptotics and the ergodic attractor in the bounded case} \label{ss:ergbounded} In \cite{Gallay:01} we introduced the notion of the {\it $\omega$-limit set on average}, denoted by $\bar{\omega}(u)$, as the set of $x \in \tilde{\mathcal{B}}$ such that $u(t)$ converges to $x$ for non-zero density of times $t$. The physical meaning of $\bar{\omega}(u)$ is as follows: if we start from $u$ and wait long enough, we are going to observe only $v \in \bar{\omega}(u)$. Even though there may be $v \in \omega(u) \setminus \bar{\omega}(u)$, any neighbourhood of such $v$ is visited only for zero density of times, thus effectively non-observable (see Lemma \ref{l:observability}). More precisely, in the time-periodic case (i.e. for the map $T$), we set
\begin{equation}
\bar{\omega}(u)=\left\lbrace v \: : \: \limsup_{n \rightarrow \infty}\frac{1}{n}\sum_{k=0}^{n-1} \mathbf{1}_{\mathcal{U}}(T^ku) >0 \text{ for all open neigborhoods }\mathcal{U} \text{ of }x \right\rbrace, \notag
\end{equation}
\noindent and in the autonomous case (i.e. for a semiflow $T(t)$), we have (using the notation $T(t)u=u(t)$):
\begin{equation}
\bar{\omega}(u)=\left\lbrace v \: : \: \limsup_{T \rightarrow \infty}\frac{1}{T}\int_{0}^T \mathbf{1}_{\mathcal{U}}(T(t)u)dt >0 \text{ for all open neigborhoods }\mathcal{U} \text{ of }x \right\rbrace. \notag
\end{equation}
The following lemma shows that the properties of $\bar{\omega}(u)$ reflect those of ${\omega}(u)$:
\begin{lemma} \label{l:onaverage}
	The set $\bar{\omega}(u)$ is non-empty, compact, $T$- (resp. $T(t)$-) invariant, and $\bar{\omega}(u) \subset \omega(u)$.
\end{lemma}
The proof is in \cite{Gallay:01}, Proposition 5.4 (for the semiflow case, the map case is analogous). We also give in \cite{Gallay:01} an example of (\ref{r:main}) for which $\omega(u) \setminus \bar{\omega}(u) \neq \emptyset$ (for example, consider $u$ whose $\omega(u)$ consists of exactly two equilibria and their two heteroclinic connections. Then $\bar{\omega}(u)$ is the two equilibria).

\begin{lemma} \label{l:subset}
	We have that $\mathcal{E}= Cl \left( \cup_{u \in \tilde{\mathcal{B}}}\bar{\omega}(u) \right)$.
\end{lemma}

\begin{proof}
	We first show that $\bar{\omega}(u) \subset \mathcal{E}$.
	We give the proof only in the time-periodic case, as the autonomous case is analogous. Let $v \in \bar{\omega}(u)$, and let $\mathcal{U}_m$ be a $1/m$-open ball around $v$. By definition of $\bar{\omega}(u)$, we can find a subsequence $n_j$ such that 
	\begin{equation}
	\lim_{j \rightarrow \infty}\frac{1}{n_j}\sum_{k=1}^{n_j}\mathbf{1}_{\mathcal{U}_m}(T^ku) = \kappa > 0. \label{r:ge0}
	\end{equation}
	Consider Borel probability measures $\mu_n = \frac{1}{n}\sum_{k=1}^{n}\delta_{T^ku}$, where $\delta_u$ is the Dirac probability measure concentrated at $u$. By compactness, we can find a weak$^*$-convergent subsequence $n'_j$ of $n_j$ such that $\mu_{n'_j}$ weak$^*$-converges to $\mu$, which is then $T$-invariant. By (\ref{r:ge0}) we have $\lim_{j \rightarrow \infty}\mu_{n'_j}(\mathcal{U}_m) = \kappa$, thus by the well-known property of weak$^*$ convergence \cite{Walters}, $\mu(\mathcal{U}_{2m}) \geq \mu (\bar{\mathcal{U}}_m) \geq \kappa > 0$. We can thus find $v_m \in \supp\mu \subset \mathcal{E}$ which is in $U_{2m}$. We repeat this for all $m \in \mathbb{N}$ and obtain $v_m \in \mathcal{E}$ such that $\lim_{m \rightarrow \infty}v_m=v$. However, by Lemma \ref{l:properties}, (ii), $\mathcal{E}$ is closed, thus $v \in \mathcal{E}$.
	
	To show the other direction, note that by the Birkhoff ergodic theorem applied to $\bf{1}_{\mathcal{U}}$ for each open set $\mathcal{U}$ in a chosen countable basis of open sets, we obtain that for any $T$-ergodic measure, the set of $u$ such that $u \in \bar{\omega}(u)$ has full measure, thus it must be dense in $\mathcal{E}$.
\end{proof}

In particular, $\mathcal{E}$ contains all uniformly recurrent $u$ (see \cite{Furstenberg:81}), as for uniformly recurrent $u$, by definition $u \in \bar{\omega}(u)$. 

We conclude the subsection with a statement on "observability" of $\mathcal{E}$.

\begin{lemma} \label{l:observability}
	Let $U$ be an open neighbourhood of $\mathcal{E}$ in the time-periodic (resp. autonomous) case. Then for any $u \in \tilde{\mathcal{B}}$, we have $\lim_{n \rightarrow \infty}\frac{1}{n}\sum_{k=0}^{n-1}\mathbf{1}_U(T^ku)=1$ (resp. $\lim_{T \rightarrow \infty}\frac{1}{T}\int_0^T\mathbf{1}_U(T(t)u)dt=1$).
\end{lemma}

\begin{proof}
	In \cite{Gallay:01}, Proposition 5.3, this was shown for any open neighbourhood $U$ of $\bar{\omega}(u)$. The claim now follows from Lemma \ref{l:subset}.
\end{proof}

\subsection{Asymptotics and the ergodic attractor in the extended case} \label{ss:ergextended}

We now consider the E/P or E/A case, or more generally a compact, metric $\tilde{\mathcal{B}}$, and two commuting continuous maps $T,S$ on $\tilde{\mathcal{B}}$ (respectively a commuting continuous semiflow $T(t)$ and a continuous map $S$), where $\mathcal{E}$ is the union of supports of all $T,S$-invariant (resp. $T(t),S$-invariant) Borel probability measures. Then $\mathcal{E}$ contains "space-time observable" orbits in the following sense:

\begin{lemma}
	Let $\mathcal{U}$ be an open neighbourhood of $\mathcal{E}$ in the time-periodic (resp. autonomous) case. Then for any $u \in \tilde{\mathcal{B}}$, we have that $$\lim_{n \rightarrow \infty}\frac{1}{n^2}\sum_{j,k=0}^{n-1}\mathbf{1}_{\mathcal{U}}(S^jT^ku)=1,$$ 
	respectively
	$$\lim_{n \rightarrow \infty}\frac{1}{n^2}\sum_{j=0}^{n-1}\int_0^n\mathbf{1}_{\mathcal{U}}(S^jT(t)u)dt=1.$$
\end{lemma}	

\begin{proof}
	See \cite{Slijepcevic15}, Proposition 4.
\end{proof}

\section*{Acknowledgement}

The author wants to thank Pavol Brunovsky and Peter Pol\'{a}\v{c}ik for useful discussions and for pointing out the very relevant results of I. Tere\v{s}\v{c}ak. This study was partially funded by the Croatian Science Foundation, the grant No IP-2014-09-2285.

{\small
{\em Authors' addresses}:
{\em Sini\v{s}a Slijep\v{c}evi\'{c}}, Department of Mathematics, Bijeni\v{c}ka 30, University of Zagreb, Croatia
 e-mail: \texttt{slijepce@\allowbreak math.hr}.

}


{\small
\begin{thebibliography}{999}
\bibitem{Angenent:88} S. Angenent, The zero set of a solution of a parabolic
equation, J. reine angew. Math. \textbf{390} (1988), 79-96. 

\bibitem{Baesens:05} C. Baesens, Spatially extended systems with monotone dynamics (continuous time), in {Dynamics of Coupled Map Lattices and of Related Spatially Extended Systems}, Lecture Notes in Physics Vol. 671, Springer (2005), p241-263.

\bibitem{Bourgain:96} J. Bourgain, Invariant measures for the 2D defocusing nonlinear Schr\"odinger equation, Comm. Math. Phys. {\bf 176} (1996), 421-445.

\bibitem{Blumenthal:15} A. Blumenthal and Lai-Sang Young, Entropy, volume growth and SRB measures for Banach space mappings, preprint, ArXiv1510:04312.

\bibitem{Chen:96} C.-Y. Chen and P. Pol\'{a}\v{c}ik, Asymptotic periodicity of positive solutions of reaction diffusion equations on a ball, J. reine angew. Math., {\bf 472} (1996), 17-51.

\bibitem{Chen:98} X.-Y. Chen, A strong unique continuation theorem for parabolic equations, Math. Annalen {\bf 311} (1998), 603-630.

\bibitem{Coutinho:05} R. Coutinho, B. Fernandez, Ed.,  {\it Dynamics of Coupled Map Lattices and of Related Spatially Extended Systems}, Lecture Notes in Physics Vol. 671, Springer (2005).

\bibitem{Durrett:05} R. Durrett, Probability: Theory and Examples, 3rd edition, Brooks/Cole, Belmont, CA, 2005.

\bibitem{E:00} Weinan E, K. Khanin, A. Mazel and Y. Sinai, Invariant measures for Burgers equation with stochastic forcing, Ann. Math. {\bf 151} (2000), 877-960.

\bibitem{Eckmann98} J.-P. Eckmann and J. Rougemont, Coarsening by
Ginzburg-Landau dynamics, Comm. Math.\ Phys. \textbf{199} (1998), 441-470. 

\bibitem{Feireisl:00} E. Feireisl and P. Pol\'{a}\v{c}ik, Structure of periodic solutions and asymptotic behavior for time-periodic reaction-diffusion equations on R, Adv. in Diff. Equations {\bf 5} (2000), 583-622.

\bibitem{Fiedler89} B. Fiedler, J. Mallet-Paret, A Poincar\'{e}-Bendixson
theorem for scalar reaction diffusion equations, Arch. Rational Mech.\ Anal. 
\textbf{107}\ (1989), 325-345. 

\bibitem{Fiedler:92} B. Fiedler, B. Sanstede, Dynamics of periodically forced parabolic equations on the circle, Ergodic Theory Dynamical Systems {\bf 12} (1992), 559-571.

\bibitem{Fiedler:12} B. Fiedler, C: Rocha, M. Wolfrum, Sturm global attractors for $S^1$-equivariant parabolic equations, Netw. Heterog. Media {\bf 7} (2012), 617-659.

\bibitem{Fiedler:14} B. Fiedler, C. Rocha, Nonlinear Sturm global attractors: unstable manifold decompositions as regular CW-complexes, Disc. Cont. Dyn. Sys. {\bf 34} (2014), 5099-5122.

\bibitem{Furstenberg:81} H. Furstenberg, Recurrence in Ergodic Theory and Combinatorial Number Theory, Princeton University Press, 1981.

\bibitem{Gallay:01} Th.\ Gallay and S. Slijep\v{c}evi\'{c}, Energy flow in
formally gradient partial differential equations on unbounded domains. J.
Dynam. Differential Equations 13 (2001), 757-789. 

\bibitem{Gallay:12} Th.\ Gallay and S. Slijep\v{c}evi\'{c}, Distribution of
Energy and Convergence to Equilibria in Extended Dissipative Systems,
to appear in J. Dynam. Differential Equations.

\bibitem{Henry} D. Henry, Geometric Theory of Semilinear Parabolic Equations, Lecture Notes in Mathematics {\bf 840}, Springer-Verlag, Berlin, 1981.

\bibitem{Hirsch:88} M. W. Hirsch, Stability and convergence in strongly monotone dynamical systems, J. Reine Angew. Math. {\bf 383} (1988), 1-53.

\bibitem{Joly10} R.\ Joly, G. Raugel, Generic Morse-Smale property for the
parabolic equation on the circle, Ann. Inst. H. Poincar\'{e}\ 27 (2010), 1397-1440. 

\bibitem{Lunardi:95} A. Lunardi, Analytic Semigroups and Optimal Regularity in Parabolic Problems, Birkh\"{a}user, Berlin (1995).

\bibitem{Katok95} A. Katok and B. Hasselblatt, \textit{Introduction to the
modern theory of dynamical systems}, Cambridge University Press, 1995. 

\bibitem{Matano:82} H. Matano, Nonincrease of the lap-number of a solution for a one-dimensional semilinear parabolic equation, J. Fac. Sci. Univ. Tokyo Sect. IA Math. {\bf 29} (2) (1982), 401-440.

\bibitem{Mielke09} A. Mielke and S. Zelik, Multi-pulse evolution and
space-time chaos in dissipative systems, Mem. Amer. Math. Soc 198 (2009), no
925. 

\bibitem{Miranville08} A. Miranville, S. Zelik, Attractors for dissipative
partial differential equations in bounded and unbounded domains, Handbook of
differential equations:\ evolutionary equations.\ Vol. IV, 103-200,
Elsevier/North-Holland, Amsterdam, 2008. 

\bibitem{Nahmod:12} A. R. Nahmod, N. Pavlovi\'{c} and G Staffilani, Almost sure existence of global weak solutions for super-critical Navier-Stokes equations, SIAM J. Math. Anal. {\bf 45} (2013), 3431-3452.

\bibitem{Polacik:02} P. Pol\'{a}\v{c}ik, Parabolic equations: asymptotic behavior and dynamics on invariant manifolds, Handbook of Dynamical Systems, Vol. 2. 835-884, Elsevier/North-Holland, Amsterdam, 2002.

\bibitem{Polacik:15} P. Pol\'{a}\v{c}ik, Examples of bounded solutions with nonstationary limit profiles for semilinear heat equation on $\mathbb{R}$, J. Evol. Equ. {\bf 15} (2015), 281-307.

\bibitem{Raugel:02} G. Raugel, Global attractors in partial differential equations, Handbook of Dynamicsl Systems, Vol. 2. 835-884, Elsevier/North-Holland, Amsterdam, 2002.

\bibitem{Sandstede:91} B. Sandstede, Asymptotic behavior of solutions of nonautonomous scalar reaction-diffusion equations, International Conference on Differential Equations (Barcelona 1991), World Scientific, 1993, 888-892.

\bibitem{Sinai:91} Ya. Sinai, Two results concerning asymptotic behavior of solutions of the Burgers equation with force, J. Stat. Phys. {\bf 64} (1991), 1-12.

\bibitem{Sinai:96} Ya. Sinai, Burgers system driven by a periodic stochastic flow, in {\it It\^{o}'s Stochastic Calculus and Probability Theory}, 347-353 Springer Verlag, New-York, 1996.

\bibitem{Sinai:98} Ya. Sinai, Asymptotic behavior of solutions of 1D-Burgers equation with quasi-periodic forcing, Topol. Methods Nonlinear Anal. {\bf 11} (1998), 219-226.

\bibitem{Slijepcevic:99} S. Slijep\v{c}evi\'{c}, Gradient dynamics of Frenkel-Kontorova models and twist maps, PhD thesis, University of Cambridge (1999).

\bibitem{Slijepcevic00} S.\ Slijep\v{c}evi\'{c}, Extended gradient systems:\
dimension one, Discrete Contin. Dyn. Syst. 6 (2000), 503-518. 

\bibitem{Slijepcevic13} S.\ Slijep\v{c}evi\'{c}, The energy flow of discrete
extended gradient systems, Nonlinearity 26 (2013), 2051-2079. 

\bibitem{Slijepcevic13a} S. Slijep\v{c}evi\'{c}, Entropy of scalar reaction-diffusion equations, Math. Bohemica {\bf 139} (2014), 597-605.

\bibitem{Slijepcevic13b} S. Slijep\v{c}evi\'{c}, The Aubry-Mather theorem
for driven generalized elastic chains, Discrete Contin. Dyn. Syst. 34 (2014), 2983-3011.

\bibitem{Slijepcevic15} S. Slijep\v{c}evi\'{c}, Stability of synchronization in dissipatively driven Frenkel-Kontorova models, Chaos, {\bf 25} (2015), pp083108.

\bibitem{Terescak:94} I. Tere\v{s}\v{c}ak, Dynamical systems with discrete Lyapunov functionals, Ph. D. thesis, Comenius Unviersity (1994).

\bibitem{Toom:90} A. Toom, N. Vasilyev, O. Stavskaya, L. Mityushin, G. Kurdyumov and S. Pirogov, Discrete local Markov systems, in R. Dobrushin, V. Kryukov, and A. Toom, Ed., {\it Stochastic Cellular Systems: Ergodicity, Memory, Morphogenesis}, Machester University Press, Manchester, 1990.

\bibitem{Turaev10} D. Turaev and S. Zelik, Analytical proof of space-time
chaos in Ginzburg-Landau equations, Discrete and Contin. Dyn. Syst. 28
(2010), 1713-1751. 

\bibitem{Walters} P. Walters, An Introduction to Ergodic Theory, Springer, 2000.

\bibitem{Zelenyak:97} T. I. Zelenyak, M. M. Lavrentiev and M. P. Vishnevskii, Qualitative Theory of Parabolic Equations, Part I, VSP (1997).

\bibitem{Zelik:03} S. Zelik, Formally gradient reaction-diffusion systems in $%
\mathbb{R}^{n}$ have zero spatio-temporal topological entropy,\ Discrete
Contin.\ Dyn. Syst. (2003), suppl., 960-966.

\end{thebibliography}}
\end{document}